\documentclass[final]{siamltex}
\usepackage{graphicx,amsmath,amsbsy,amssymb,color}

\def \z {\bf z}

\def\N{{\mathcal N}}

\newtheorem{lmm}{Lemma}[section]
\newtheorem{thm}{Theorem}[section]
\newtheorem{ex}{Example}[section]

\graphicspath{{fig/}}

\newtheorem{remark}{Remark}[section]

\title{A $C^0$ Linear Finite Element Method for a Second Order Elliptic Equation in Non-Divergence Form with Cordes Coefficients}

\author{Minqiang Xu\thanks{College of Science, Zhejiang University of Technology, Hangzhou, 310023, P.R. China, and School of Data and Computer Science, Sun Yat-sen University, Guangzhou 510275, P.R. China. Email: xumq9@mail2.sysu.edu.cn. The research of this author was supported in part by NSFC Grant 12271049, the General
Scientific Research Projects of Zhejiang Education Department (Y202147013) and the Opening Project of Guangdong Province Key Laboratory of Computational Science at the Sun Yat-sen University(2021008).}
\and Runchang Lin\thanks{Department of Mathematics and Physics, Texas A\&M International University, Laredo, Texas 78041, USA. Email: rlin@tamiu.edu.}
\and Qingsong Zou\thanks{Corresponding author.
School of Computer Science and Engineering, and Guangdong Province Key
Laboratory of Computational Science, Sun Yat-sen University, Guangzhou 510275, China.
Email: mcszqs@mail.sysu.edu.cn. The research of this author was supported in part by NSFC Grant 12071496,
Guangdong Provincial NSF Grant 2017B030311001, and Guangdong Province Key Laboratory of Computational
Science at the Sun Yat-sen University(2020B1212060032).}}
\begin{document}

\maketitle

%
%
\medskip

\begin{abstract}
In this paper, we develop a gradient recovery based linear (GRBL) finite element method (FEM) and a Hessian recovery based linear (HRBL) FEM for second order elliptic equations in non-divergence form. The elliptic equation is casted into a symmetric non-divergence weak formulation, in which second order derivatives of the unknown function are involved. We use gradient and Hessian recovery operators to calculate the second order derivatives of linear finite element approximations. Although, thanks to low degrees of freedom (DOF) of linear elements, the implementation of the proposed schemes is easy and straightforward, the performances of the methods are competitive. The unique solvability and the $H^2$ seminorm error estimate of the GRBL scheme are rigorously proved. Optimal error estimates in both the $L^2$ norm and the $H^1$ seminorm have been proved when the coefficient is diagonal, which have been confirmed by numerical experiments. Superconvergence in errors has also been observed. Moreover, our methods can handle computational domains with curved boundaries without loss of accuracy from approximation of boundaries. Finally, the proposed numerical methods have been successfully applied to solve fully nonlinear Monge-Amp\`{e}re equations.

\vskip .3cm

\end{abstract}


\section{Introduction}
In this paper, we develop and analyze a $C^0$ linear FEM for the following second-order linear elliptic partial differential equation (PDE) in non-divergence form:
\begin{eqnarray}\label{eqn:para}
\left \{
\begin{array}{lll}
\mathcal{L}u=f \quad &\mbox{in} \quad~\Omega,\\
~~u=0  \quad &\mbox{on} \quad \partial\Omega,
\end{array}
\right.
\end{eqnarray}
where $\Omega \subset \mathbb{R}^d$ is a bounded open convex domain with boundary $\partial \Omega$, $f\in L^2(\Omega)$ is a given function, and the differential operator $\mathcal{L}$ has a non-divergence form
\begin{equation}\label{nondivergence operator}
  \mathcal{L}v=A:D^2v=\sum_{j,k=1}^{d}a_{jk}\partial_{jk}^2v, \quad \forall v\in V:=H_0^1(\Omega)\cap H^2(\Omega).
\end{equation}
Here and in the rest of this paper, standard definitions and notations of Sobolev spaces are used \cite{Adama}. We suppose that the coefficient tensor $A = (a_{ij})_{d\times d}$ is symmetric and uniformly bounded. Assume further that $A$ is positive definite; namely, there exist positive constants $\alpha,\beta$ such that
\begin{equation}\label{posideft}
 \alpha\xi^{T}\xi \leq\xi^{T}A(x)\xi\leq \beta \xi^{T}\xi, \quad \forall~\xi\in \mathbb{R}^d,~x\in\Omega.
\end{equation}
In addition, we assume that the coefficient tensor satisfies the Cordes condition; i.e. there exists an $\epsilon\in[0,1]$ such that
\begin{equation}\label{Cordes}
 |A|^2/(\text{tr} A)^2\leq 1/(d-1+\epsilon),
\end{equation}
where $|A|^2=\sum_{i,j=1}^{d}a_{ij}^2$. It has been proven in \cite{Smears} that the condition \eqref{Cordes} can be derived from the positive definiteness condition \eqref{posideft} for two dimensional problems. But, in three dimensional cases, the Cordes condition \eqref{Cordes} is essential; the PDE may be ill-posed in absence of this condition.

Problem \eqref{eqn:para} arises in many fields, such as stochastic processes and game theory \cite{Fleming}. The non-divergence equations are also frequently found in linearizations of second order fully nonlinear differential equations, such as the Hamilton-Jacobi-Bellman equation and the Monge-Amp\`{e}re equation (cf., e.g., \cite{Brenner,Neilan2}). In many important applications, the coefficients are hardly smooth, or even discontinuous, so that the differential equations cannot be written in divergence forms. On the other hand, compared with studies for elliptic problems in divergence form, the literature on numerical analysis of differential equations in non-divergence form is limited. Therefore, it is crucial to develop efficient numerical methods for the problem \eqref{eqn:para} with rough coefficients to accommodate its wide application.

In this paper, we will consider the problem \eqref{eqn:para} in two dimensional cases. If the coefficient matrix $A\in [C^1(\Omega)]^{2\times2}$, then (1.1) can be recast into the following divergence form:
\begin{equation}\label{divergence form}
  \nabla\cdot(A\nabla u)-(\nabla\cdot A)\cdot \nabla u=f.
\end{equation}
A weak formulation of the problem \eqref{divergence form} is to find $u\in H^1(\Omega)$ such that
\begin{equation}\label{standard formulation}
 -\int_{\Omega} (A\nabla u)\cdot\nabla v-\int_{\Omega}(\nabla\cdot A)\cdot(\nabla u) v=\int_{\Omega} fv, \quad \forall v\in H_0^1(\Omega).
\end{equation}
Therefore, standard Lagrange finite elements can be applied to discrete formulation \eqref{standard formulation}. But for a non-divergence form (1.1), the formulation \eqref{standard formulation} may fail to work. To circumvent this difficulty, many numerical approaches have been developed; see, e.g., \cite{Blechschmidt2019, Brenner2020, Cordes1956, Dedner, Feng, Gallistl, Kawecki20191, Lakkis2019, Lakkisand, Mu, Neilan1, Nochetto, Smears, Smears1, Wang, Zhu2020} and the references therein for an incomplete list of references. Among these methods, we are interested in three approaches.

The first approach is based on an asymmetric form, which reads: finding $u\in H^2(\Omega)$ such that
\begin{equation}\label{asymmetric formulation}
  \int_{\Omega} (A:D^2u) v=\int_{\Omega}fv, \quad \forall v\in H^1_0(\Omega).
\end{equation}
To discretize formulation \eqref{asymmetric formulation}, Wang et al. introduced and analyzed a primal-dual weak Galerkin (WG) method \cite{Wang}, which characterizes the numerical solution as a minimization of a nonnegative quadratic functional with constraints. This method involves second derivatives of test and trial functions, which means polynomials of degree at least two are required for the finite element space. Lakkis et al. provided a nonconforming FEM by introducing finite element Hessian \cite{Lakkisand}.
Stability and convergence of the method were provided in the case of quadratic or higher degree elements.

The second approach involves a fourth-order variational form of the non-divergence equation \eqref{eqn:para}. Smears and S\"{u}li \cite{Smears} designed an $hp$-version discontinuous Galerkin (DG) FEM based on the formulation
\begin{equation}\label{formulation 2}
  (\gamma A:D^2u,\Delta v)_{\Omega}=(\gamma f,\Delta v)_{\Omega}, \quad \forall~v\in H^2(\Omega),
\end{equation}
which was the first contribution to the non-divergence equations \eqref{eqn:para} with Cordes coefficients. The stability of the presented scheme was shown by applying a discrete Miranda-Talenti estimate. Feng et al. \cite{Feng2018} utilized continuous Lagrange finite elements to discrete scheme \eqref{formulation 2} and proved the well-posedness of the proposed scheme using a discrete inf-sup condition under the assumption that coefficients are continuous. Neilan et al. \cite{Feng} proposed and investigated a $C^0$ DG method. They used an interior penalty term from the jump of flux across interior element edges, which can be obtained by applying DG integration by parts formula to the first term of the formulation \eqref{standard formulation}.

The third approach is based on a symmetric form from the least-squares technique, which has widely applications in
scientific computing (see, e.g., \cite{refJiaXu,refXZE3,refXZE4}). It reads: seeking $u\in H^2 (\Omega)$ such that
\begin{equation}\label{least-squares formulation 1}
    (A:D^2u,A:D^2v)_{\Omega}=(f,A:D^2v)_{\Omega}, \quad \forall~v\in H^2(\Omega).
\end{equation}
The formulation \eqref{least-squares formulation 1} can be obtained from minimizing the functional $\|A:D^2u-f\|^2_{0}$, for which  $H^2$ elements are usually required. Gallistl \cite{Gallistl} applied a conforming mixed FEM (MFEM) for the numerical approximation. Adaptive algorithms were also discussed. Ye et al. \cite{Mu} presented a nonconforming FEM with interior penalty term. There are many other approaches for solving this classic problem; e.g. the vanishing moment method \cite{FN2008}, the Alexandroff-Bakelman-Pucci (ABP) method \cite{Nochetto}, the tailored nonconforming FEM \cite{BRW2020}, etc. Recently, Kawecki \cite{Kawecki20191} extended the DG technique to curved domains.


Recently, some differential operator recovery based linear FEMs have been proposed to solve high order partial differential equations (see, e.g., \cite{Guo2017,XuGuoZou}). The main purpose of this work is to design linear FEMs for problems \eqref{eqn:para}. A challenge of applying low degree elements is in the calculation of second order derivatives of the linear finite element approximation. To overcome this difficulty, we adopt a gradient recovery operator $G_h$ (see, e.g., \cite{Ainsworth, Bank1, Bank2}) to lift the discontinuous piecewise constant $\nabla v_h$ to a continuous piecewise linear function $G_hv_h$, such that differentiation $DG_hv_h$ is possible. Alternatively, we can apply Hessian recovery operators $H_h$ (see, e.g., \cite{GuoZhang}) to discretize the second order differential operator $D^2$ directly. Once the recovery operators $DG_h$ or $H_h$ have been constructed, numerical schemes for \eqref{eqn:para} can be designed by applying least-squares weak formulation \eqref{least-squares formulation 1}. We shall remark that the difference operator $DG_h$ is asymmetric in general. Consequently, a direct application of $DG_h$ to \eqref{least-squares formulation 1} may lead to an instable numerical scheme. To ensure stability, the rotation of the recovery gradient may be included as a penalty in the scheme.

Comparing to other techniques for the non-divergence form \eqref{eqn:para}, the proposed methods have two advantages. First, linear elements induce {\color{red}fewer} DOFs in comparison to $C^1$ and/or mixed elements, which hence leads to more convenient implementation and less cost in computation. For example, in Table 1.1, we demonstrate local and global DOFs from different methods for a benchmark problem on a square domain and a uniform mesh with $2N^2$ triangles. In particular, the DG method in \cite{Dedner}, the MFEM in \cite{Gallistl}, the WG method in \cite{Wang}, and the internal penalty FEM (IPFEM) in \cite{Mu} are included in the Table. The total DOFs of the GRBL and HRBL FEMs are both $(N+1)^2$, which are the smallest in these methods.
\begin{table}[ht]
\centering
\caption{Comparison on a uniform triangular mesh}
\label{Tab:01}
\begin{tabular}{lllll}
\hline
~\text{Methods}  &~~~~\quad~~~\text{Elements}                           &\text{Local DOFs}             &~~\quad\text{Global DOFs}      \\
\hline
\quad~~\text{DG}         &~~~~~$P_{1}(T)\times [P_{1}(T)]^2$                &~~~~~~9                              &~~~\quad\quad$18N^2$   \\
\quad\text{MFEM}       &~~~~~$P_{1}(T)\times [P_{1}(T)]^2$                  &~~~~~~9                              &~~~\quad $3(N+1)^2$                   \\
\quad~~\text{WG}       &$P_{2}(T)\times P_{2}(e)\times[P_{1}(e)]^2$         &~~~~~15                             &$2(2N+1)^2+2(N+1)^2$              \\
~~~\text{IPFEM}        &~~~~~~\quad~~$P_{2}(T)$                             &~~~~~~6                              &~~~~\quad$(2N+1)^2$ \\
\text{GRBL/HRBL}      &~~~~~~\quad~~$P_{1}(T)$                             &~~~~~~3                              &~~~~~\quad$(N+1)^2$\\
\hline
\end{tabular}
\end{table}
Second, the recovery operators $G_h$ and $H_h$ can be defined on a general unstructured grid. Thus the numerical algorithm can be applied for problems on domains with arbitrary geometries.
Meanwhile, our proposed numerical schemes have nice convergence properties. Under the assumption of $H^3(\Omega)$ regularity of the exact solution, it is observed that numerical errors measured in $L^2$ and $H^1$ norms converge optimally on unstructured grids, which are of second and first orders, respectively. These convergence rates are competitive to the rates of other methods in the literature. For example, for both the WG method in \cite{Wang} and the internal penalty method in \cite{Mu}, the convergence rates of numerical errors in $L^2$ norm are of second order when quadratic elements are used, which are not optimal. On the other hand, superconvergence phenomena are also observed in our numerical experiments. In particular, numerical errors of the HRBL have a convergence rate of 1.5 when they are measured in the $H^2$ seminorm. The recovered gradient of both schemes converges in second-order. In addition, numerical experiments show that, when inexact approximations of curved boundaries are employed, the proposed methods capture optimal second order convergence rate as well. Even for less smooth solutions with only $H^{2+\tau}(\tau>0)$ regularity, the proposed schemes using linear elements can still achieve same convergence rates as those obtained from the WG method using quadratic elements. Finally, as an application, we have applied the recovery based linear FEM to solve fully nonlinear Monge-Amp\`{e}re equations. A convex solution with optimal convergence rates is obtained.

In this paper, theoretical investigation for the GRBL scheme has been developed. Error estimation in the $H^2$ seminorm converges with linear convergence order has been given under the assumption of $H^3$ regularity of the exact solution and sufficient regularity of the grid. Moreover, in the special case that $A=\alpha I$, we can prove superconvergence of the recovered gradient by applying the Aubin-Nitsche technique. Consequently, an optimal $L^2$ error estimate can be obtained by the discrete Poincar\'{e} inequality.

This paper is organized as follows. In Section 2, notations and some preliminary results on the gradient and Hessian recovery operators are introduced. In Section 3, we first introduce the GRBL and HRBL FEMs for problem $\eqref{eqn:para}$. In section 4, the stability of the GRBL FEM is proven. Moreover, some optimal error estimates, including estimations in $H^2$ seminorm, recovered gradient seminorm, and $L^2$ norm, are established. In Section 5, we introduce an application of the proposed method to fully nonlinear Monge-Amp\`{e}re equations. In Section 6, some typical (including benchmark) numerical experiments are presented to demonstrate the effectiveness of the new numerical methods.

\section{Preliminary results}

In this paper, we use $C$ to denote a generic positive constant independent of data of the PDE and mesh size, which may be different at different occurrences. For convenience, we write $x\lesssim y$ provided $x \leq Cy$ for some constants $C$, and $x\sim y$ if both $x\lesssim y$ and $y\lesssim x$ hold. Standard definitions and notations for Sobolev spaces are used. In particular, $(\cdot,\cdot)$ is the $L^2$-inner product, and $\|\cdot \|_{i}$ and $|\cdot|_{i}$ are the norm and seminorm in $H^{i}(\Omega)$, respectively.

For simplicity of presentation, we focus our attention on the two-dimensional case. Let $\mathcal{T}_{h}$ be a regular triangulation of the domain $\Omega$ with mesh-size $h$. We use $\mathcal{N}_h$ to denote the set of vertices of $\mathcal{T}_{h}$. We denote by $V_h$ the standard $C^0$ linear finite element space associated with $\mathcal{T}_{h}$ and define $V_h^0=\{v_h\in V_h : v_h|_{\partial \Omega}=0\}$. For each vertex $\z$ in the triangulation, we define the element patch and control volume of $\z$ as $\omega_{\z}=\cup\{\tau \in \mathcal{T}_{h}: \z\in\bar{\tau}\}$ and $V_{\z}$ (see Figure 2.1).

\begin{figure}[!h]
    \centering
    {\includegraphics[width=0.35\textwidth]{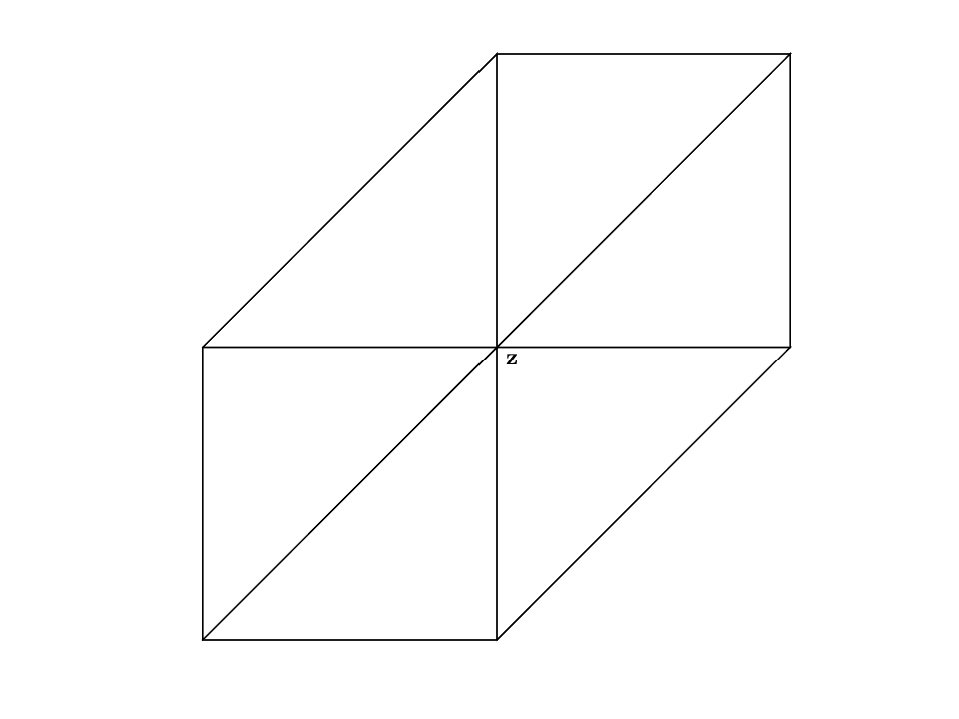}}
    {\includegraphics[width=0.35\textwidth]{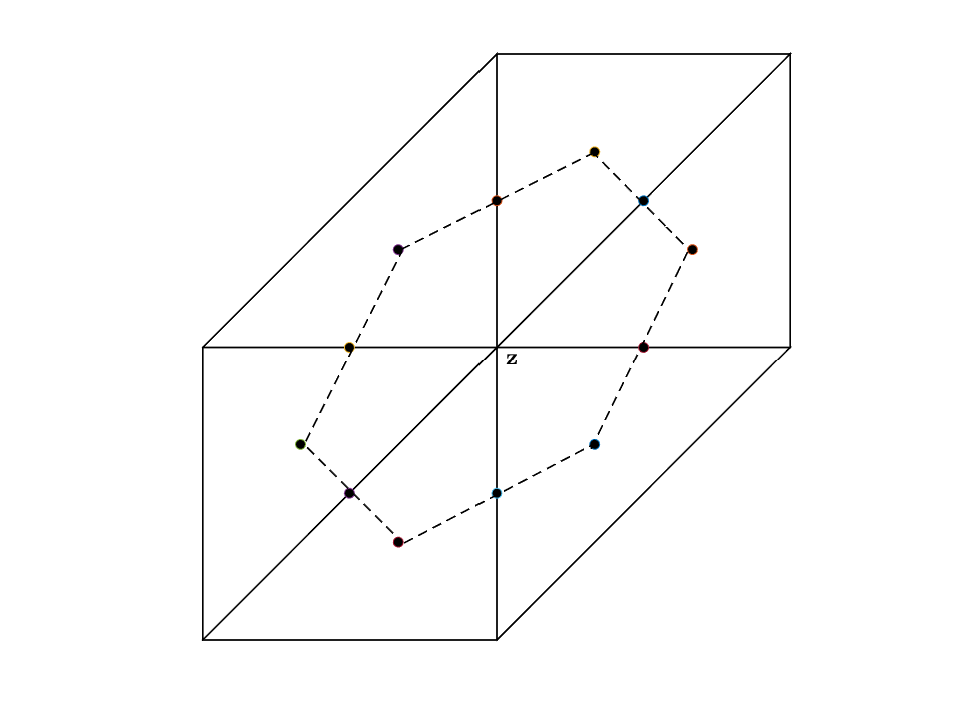}}
    \caption{(a) local patch $\omega_{\z}$, and (b) control volume $V_{\z}$ on uniform mesh.}
    \label{fig:dualandpatch}
\end{figure}

\subsection{Calculation of first derivatives}
In this subsection, we introduce a reconstructed gradient $G_hv_h$ which is an improvement of a piecewise constant function $\nabla v_h$.  We shall define a gradient reconstruction operator $G_h:V_h\rightarrow V_h\times V_h$. We first determine the value of $G_h v_h$ at all vertices, and then obtain the gradient function by interpolation over the whole domain, namely,
\begin{equation*}
    G_hv_h (x,y) =\sum _{\z\in\mathcal{N}_h} G_hv_h({\z})\phi_{\z} (x,y),
\end{equation*}
where $\phi_{\z}$ is the linear nodal shape function of vertex $\z$. There are three popular ways to determine $G_hv_h({\z})$ at a vertex ${\z} \in\mathcal{N}_h$, which are specified in below.

(a) Weighted average: define
\begin{equation}\label{2.3}
  G_hv_h({\z})=\frac{1}{|\omega_{\z}|}\int_{\omega_{\z}}\nabla v_h \; dxdy.
\end{equation}

(b) Recovery techniques: using a local discrete least-squares fitting operator to smooth the gradient. The ZZ approach proposed by Zienkiewicz and Zhu \cite{Zienkiewicz}
and polynomial preserving recovery (PPR) proposed by Naga and Zhang \cite{Zhang} are frequently used operators in post-processing technology. Specifically, they are defined as follows.

ZZ: seeking two linear polynomials $p_l\in \mathcal{P}_1(\omega_{\z})$ satisfying
\begin{equation}
  \sum_{i=1}^{m}[p_l(x_i,y_i)-\partial_lv_h(x_i,y_i)]q(x_i,y_i)=0, \quad \forall~q\in \mathcal{P}_1(\omega_{\z}),
\end{equation}
where $l = x$ or $y$, and $(x_i,y_i)$, $i=1,2\cdots,m$, are $m$ given points in $\omega_{\z}$. Then the nodal value of $G_hv_h$ can be defined as
\begin{equation*}
  G_hv_h({\z})=(p_x({\z}),p_y({\z})).
\end{equation*}

PPR: seeking a quadratic function $p\in\mathcal{P}_2(\omega_{\z})$, such that
\begin{equation}\label{gradint recovery formula1}
 \sum_{i=1}^{m}[p(x_i,y_i)-v_h(x_i,y_i)]q(x_i,y_i)=0, \quad \forall~q \in \mathcal{P}_2(\omega_{\z}).
\end{equation}
Then the nodal value of $G_hv_h$ can be defined as
\begin{equation}\label{gradint recovery formula}
  G_hv_h({\z})=(\partial_xp({\z}),\partial_yp({\z})).
\end{equation}

(c) Green's formula: the determination of $G_hv_h(\z)$ proceeds with the help of the Green's formula,
\begin{equation}\label{Green's formula}
  \int_{V_{\z}}\partial_iv\text{d}\mathbf{x}=\int_{\partial V_{\z}}v n_i\text{d}s.
\end{equation}
Here $n_i$ is the $i$th component of the unit outward-pointing normal $\mathbf{n}$. Then the nodal value of $G_hv_h$ can be defined as
\begin{equation*}
  G_hv_h({\z})=\frac{1}{|V_{\z}|}\Big(\int_{\partial V_{\z}}v_h n_1\text{d}s,\int_{\partial V_{\z}}v_h n_2\text{d}s\Big).
\end{equation*}

\begin{remark}
    The three definitions above are equivalent on a uniform triangular mesh.
\end{remark}

We next review some properties of the gradient recovery operator. In particular, properties \eqref{Boundedness} and \eqref{Consistency} are always valid on general grids.

(a) Boundedness (cf., e.g., \cite{Naga,Zienkiewicz,Zhang}):
\begin{equation}\label{Boundedness}
  \|G_hv_h\|_{0} \lesssim |v_h|_{1}, \quad \forall v_h\in V_h.
\end{equation}

(b) Consistency (cf., e.g., \cite{Guo,Zhang}):
\begin{equation}\label{Consistency}
  \|\nabla u-G_hu_I \|_{0}\lesssim h^2 \|u\|_{3}, \quad \forall u\in H^3(\Omega),
\end{equation}
where $u_I$ is the linear interpolation of $u$ in $V_h$.

Throughout the rest of this paper, we assume that the mesh $\mathcal{T}_h$ is sufficiently regular such that the following discrete Poincar\'{e} inequality holds:
\begin{equation}\label{Poincare inequality}
  \|v_h\|_{i}\lesssim \|G_hv_h\|_{i}, \quad \forall v_h\in V_h^0, \quad i=0,1.
\end{equation}

\begin{remark} In \cite{Guo2017}, the authors proved that \eqref{Poincare inequality} is valid for some uniform meshes. Moreover, numerical results indicate that \eqref{Poincare inequality} holds for some weakly regular grids.
\end{remark}
\begin{remark} On the boundary of a domain $\Omega$, modification in the gradient recovery operator $G_h$ is necessary to maintain the superconvergence property \eqref{Consistency}, as long as numerical data nearby are available. For more details, we refer to \cite{GZZZ2016}.
\end{remark}

\subsection{Calculation of the second derivatives}
It is impossible to calculate the second derivatives of a linear finite element function directly since its gradient is piecewise constant and discontinuous across element boundaries. To overcome this difficulty, we introduce some techniques for approximating the second derivatives of linear elements in this subsection.

The first technique for approximating Hessian is derived from gradient reconstruction techniques.
As $G_hv_h$ is continuous piecewise linear, hence further differentiation $DG_hv_h$ is possible.
Therefore, the Hessian matrix of a linear function can be approximated as follows:
\begin{equation}\label{weak Hessian}
               DG_hv_h=\begin{pmatrix}
               \partial_xG_h^xv_h & \partial_xG_h^yv_h\\
               \partial_yG_h^xv_h & \partial_yG_h^yv_h \\
             \end{pmatrix}.
\end{equation}
Note that $DG_hv_h$ is piecewise constant. Therefore, $DG_hv_h$ is not well defined on the common side of two elements.

The second technique for approximating Hessian overcomes this difficulty. The basic idea of a Hessian reconstruction operator $H_h:V_h \rightarrow V_h^2\times V_h^2$ is either applying the gradient recovery operator twice or directly computing the second derivative of quadratic polynomial $p$ in \eqref{gradint recovery formula1}. That is, the nodal value of the reconstructed Hessian is determined by:
\begin{equation}\label{equ:hessian1}
  (H_hv_h)({\z})=(G_h(G_hv_h))({\z}),
\end{equation}
or
\begin{equation}\label{equ:hessian2}
(H_hv_h)({\z})
=\left(
\begin{matrix}
 H_h^{xx}v_h({\z}) &  H_h^{xy}v_h({\z})\\
  H_h^{yx}v_h({\z}) &  H_h^{yy}v_h({\z})
\end{matrix}
\right)=\left(
\begin{matrix}
 \frac{\partial^2 p}{\partial x^2}({\z}) & \frac{\partial^2 p}{\partial x\partial y}({\z}) \\
 \frac{\partial^2 p}{\partial y\partial x}({\z})&  \frac{\partial^2 p}{\partial y^2}({\z})  \\
\end{matrix}
\right).
\end{equation}
By interpolating the whole region, we obtain
$$H_hv_h=\sum_{{\z}\in\N_h} H_hv_h({\z})\phi_{{\z}}.$$

The third approach for approximating Hessian is to apply Green's formula, namely,
\begin{equation*}
  \int_{V_{\z}}\partial_{ij}v\text{d}\mathbf{x}=\int_{\partial V_{\z}}(\partial_iv) n_j\text{d}s.
\end{equation*}
Then the nodal value of $H_h^{ij}v_h$ can be defined by
\begin{equation}\label{Green's formula Hessian}
  (H_h^{ij}v_h)({\z})=\frac{1}{|V_{\z}|}\int_{\partial V_{\z}}(\partial_iv_h) n_j\text{d}s,~i,j=x,y.
\end{equation}

\begin{remark} We shall remark that the approximated Hessian derived from \eqref{equ:hessian2} satisfies the symmetric property $H_h^{xy}=H_h^{yx}$. Moreover, the discrete Laplace operator defined by $H_h^{xx} + H_h^{yy}$ on regular pattern uniform grids is the well-known five-point finite difference scheme.
\end{remark}

\section{The Recovery Based Linear Finite Element Methods}

\subsection{Algorithm}

Recall the symmetric weak formulation \eqref{least-squares formulation 1} of the problem \eqref{eqn:para}. We introduce a bilinear form
\begin{equation}\label{continuous bilinear form}
  a(v,w)=\int_\Omega (A:D^2v)\cdot(A:D^2w),~\forall~v,w\in H^2(\Omega).
\end{equation}
The formulation \eqref{least-squares formulation 1} can be written as: Finding $u\in H^2(\Omega)$ such that
\begin{equation}\label{continuous formulation}
  a(u,v)=(f,A:D^2v),~\forall~v\in H^2(\Omega),
\end{equation}
which allows naturally the application of $C^1$ finite elements. However, due to the inconvenience of construction and implementation of $C^1$ elements, alternative discretization approaches of \eqref{least-squares formulation 1} are desirable. In \cite{Gallistl}, the author designed a mixed FEM in the framework of saddle-point problems.

In this paper, we shall discretize \eqref{least-squares formulation 1} in a recovered FEM. The essential issue is the calculation of second derivatives $D^2$. Define the function space
$$V_h^{0,0}=\left \{v_h\in V_h^0 \; : \; \text{the~tangential~trace~of}~G_hv_h~\text{on}~\partial \Omega~\text{vanishes} \right\}.$$
For any given function $v_h\in V_h^{0,0}$, there are two approaches to approximate the Hessian matrix of $v_h$. The first approach uses the gradient recovery operator $G_h$ to obtain a continuous piecewise linear function $G_hv_h=(G_h^xv_h,G_h^yv_h)$,
then the Hessian matrix of $v_h$ can be approximated by differentiating $G_hv_h$. In this paper, we utilize the PPR technique \eqref{gradint recovery formula1} to produce the $G_h$. Correspondingly,
for all $v_h,w_h\in V_h$, we define a bilinear form as:
\begin{equation}\label{bilinear form}
  a_{1,h}(v_h,w_h)=\sum_{T\in \mathcal{T}_h}\int_{T}\Big[(A:DG_hv_h)\cdot(A:DG_hw_h)+(\text{rot}G_hv_h)\cdot(\text{rot}G_hw_h)\Big]\text{d}x\text{d}y,
\end{equation}
where the rotation for a two-dimensional $H^1$ vector $\mathbf{v}=(v_1,v_2)$ is defined as
\begin{equation*}
  \text{rot}\mathbf{v}=\partial_yv_1-\partial_xv_2.
\end{equation*}
We propose the first numerical method in below.

\noindent \emph{\textbf{Scheme 1}}: The gradient recovery based linear (GRBL) FEM for approximation of \eqref{eqn:para} is to find $u_h\in V_h^{0,0}$ such that
\begin{equation}\label{discrete weak formulation}
   a_{1,h}(u_h,v_h)=(f,A:DG_hv_h), \quad \forall~v_h\in V_h.
\end{equation}
Note that the second term of the bilinear form $a_{1,h}(\cdot,\cdot)$ is a penalty term to ensure the stability of the method. We accordingly define an energy seminorm as
\begin{equation*}
  \|v_h\|_{a,1}^2 = a_{1,h}(v_h,v_h)=\|A:DG_hv_h\|_0^2+\|\text{rot}G_hv_h\|_0^2, \quad \forall~v_h\in V_h^0.
\end{equation*}

Another approach of approximating the Hessian matrix is the direct application of the Hessian recovery operator. Here, we use \eqref{equ:hessian2} to obtain operator $H_h$. We can thus alternatively define a bilinear form as:
\begin{equation}\label{bilinear form 2}
  a_{2,h}(v_h,w_h)=\sum_{T\in \mathcal{T}_h}\int_{T}\Big[(A:H_hv_h)\cdot(A:H_hw_h)\Big]\text{d}x\text{d}y.
\end{equation}
The second numerical method is proposed in below.

\noindent \emph{\textbf{Scheme 2}}: The Hessian recovery based linear (HRBL) FEM for approximation of \eqref{eqn:para} is to find $u_h\in V_h^{0,0}$ such that
\begin{equation}\label{discrete weak formulation 2}
   a_{2,h}(u_h,v_h)=(f,A:H_hv_h), \quad \forall~v_h\in V_h.
\end{equation}

We shall remark that, since the finite element space $V_h$ is not in $C^1$, numerical schemes \eqref{discrete weak formulation} and \eqref{discrete weak formulation 2} are both nonconforming methods.

\subsection{Stability of scheme $\eqref{discrete weak formulation}$}
Let the space of $H^1$ vector fields with vanishing tangential trace be
\begin{equation*}
    \mathbf{W} = \left\{v\in H^1(\Omega,{\color{red}\mathbb{R}^2}):\text{the~tangential~trace~of}~v~\text{on}~\partial \Omega~\text{vanishes}\right\}.
\end{equation*}
It is useful to note that, on convex domains, the following estimate holds \cite[Theorem 2.3]{Costabel}
\begin{equation}\label{inequality cite1}
  \|D \mathbf{w}\|_{0}^2\leq \|\text{rot} \mathbf{w}\|_0^2+\|\text{div} \mathbf{w}\|_0^2, \quad \forall~\mathbf{w}\in \mathbf{W}.
\end{equation}

Following the strategy proposed in \cite{Smears}, we define the function $\gamma$ by
\begin{equation*}
  \gamma=\text{tr}(A)/|A|^2.
\end{equation*}
We have the following lemma.

\begin{lmm}
There exists a positive constant $c$ independent of $h$, such that for any $v_h\in V_h^{0,0}$, it holds
\end{lmm}
\vskip -0.5cm
\begin{equation}\label{inequality3}
   \|\gamma A:DG_hv_h\|_0^2+\|\text{rot}~G_hv_h\|_0^2\geq c\|D G_hv_h\|_0^2.
\end{equation}

\begin{proof} From \cite{Gallistl,Smears}, it holds that for any $B\in \mathbb{R}^{2\times2}$
\begin{equation}\label{inequality4}
  |(\gamma A-I):B|=|\gamma A:B-\text{tr}(B)|\leq\sqrt{1-\epsilon}|B|.
\end{equation}
By triangle inequality, one gets
\begin{equation}\label{inequality5}
  |\gamma A:B|\geq  ||\text{tr}(B)|-|(\gamma A-I):B||.
\end{equation}
We take $B=DG_hv_h$, the combination of \eqref{inequality3}, \eqref{inequality4}, and the Young inequality results in
\begin{eqnarray}\nonumber
  &&\|\gamma A:DG_hv_h\|_0^2 \geq\|(|\text{tr}(B)|-|(\gamma A-I):B|)\|_0^2\\ \nonumber
  &=&\|\text{div}(G_hv_h)\|_0^2+\|(\gamma A-I):DG_hv_h\|_0^2-2\left(|\text{div}(G_hv_h)|,|(\gamma A-I):DG_hv_h|\right)\\ \nonumber
  &\geq&(1-\beta)\|\text{div}(G_hv_h)\|_0^2+ \Big(1-\frac{1}{\beta} \Big)\|(\gamma A-I):DG_hv_h\|_0^2\\\label{inequality cite2}
  &\geq&(1-\beta)\|\text{div}(G_hv_h)\|_0^2+ \Big(1-\frac{1}{\beta} \Big)(1-\epsilon)\|DG_hv_h\|_0^2.
\end{eqnarray}
In the last inequality, a constant $0<\beta<1$ is required.

Since $G_hv_h\in {\bf W}$, by $\eqref{inequality cite1}$ and $\eqref{inequality cite2}$, a simple calculation shows that
\begin{eqnarray}\nonumber
  &&\|\gamma A:DG_hv_h\|_0^2+\|\text{rot}~G_hv_h\|_0^2\\ \nonumber
  &\geq&(1-\beta)\|DG_hv_h\|_0^2+ \Big(1-\frac{1}{\beta} \Big)(1-\epsilon)\|DG_hv_h\|_0^2+\beta\|\text{rot}~G_hv_h\|_0^2\\ \nonumber
  &\geq&\left[(1-\beta)+ \Big(1-\frac{1}{\beta}\Big)(1-\epsilon)\right]\|DG_hv_h\|_0^2\\ \nonumber
  &=&c(\epsilon,\beta)\|DG_hv_h\|_0^2.
\end{eqnarray}
Here $\beta$ is chosen to satisfy $1-\epsilon < \beta <1$, so that $c(\epsilon,\beta)>0$. This completes the proof.
\end{proof}

The proposed algorithm is well-posed.

\begin{thm}\label{thm:Stablity}
If the mesh is sufficiently regular such that \eqref{Poincare inequality} holds, then scheme $\eqref{discrete weak formulation}$ is uniquely solvable.
\end{thm}
\begin{proof}The lower bound $\eqref{inequality3}$ yields that
\begin{equation}\label{inequality6}
  \text{max}\left\{\|\gamma\|_{L^{\infty}(\Omega)}^2,1\right\}\|v_h\|_{a,1}^2\geq\|\gamma A:DG_hv_h\|_0^2+\|\text{rot}~G_hv_h\|_0^2\geq c(\epsilon,\beta)\|DG_hv_h\|_0^2,
\end{equation}
which implies that
\begin{equation}\label{inequalityleft}
  \|v_h\|_{a,1} \geq d(\gamma,\epsilon,\beta)\|D G_hv_h\|_0.
\end{equation}
Here $d(\gamma,\epsilon,\beta)=c(\epsilon,\beta)^{1/2}/\text{max}\left\{\|\gamma\|_{L^{\infty}(\Omega)},1\right\}$ is independent of mesh size $h$.

On the other hand, we have
\begin{equation}\label{inequalityright}
  \|v_h\|_{a,1}\leq (\|A\|_{\infty}^2+2)^{1/2}\|D(G_hu_h)\|_{0}.
\end{equation}
Therefore, the seminorm $\|v_h\|_{a,1}$ is equivalent to $\|D(G_hu_h)\|_0$ in $V_h^{0,0}$.
In \cite{Guo2017}, it has been proven that $\|D(G_hu_h)\|_0$ is a norm in $V_h^{0,0}$, which implies $\|u_h\|_{a,1}$ is a norm in $V_h^{0,0}$. By the Lax-Milgram theorem, there exists a unique solution to $\eqref{discrete weak formulation}$ in $V_h^{0,0}$.
\end{proof}

\section{Error Estimates}
In this section, we first develop error estimation in the $H^2$ norm. Then, we establish the $H^1$ and $L^2$ error bounds in a special case.

\subsection{Error Estimate in $H^2$ Norm}
We have an optimal error estimate.

\begin{lmm}
Let $u\in H^{2+\alpha}(\Omega)$ with $0<\alpha<1$, then the following inequality holds true for any general grids,
\begin{equation}\label{inequalityuhuI}
   \|D^2 u-DG_hu_I\|_{0}\lesssim h^\alpha\|u\|_{2+\alpha}.
\end{equation}
Consequently, it follows that
\begin{equation}\label{inequalityuI}
   \|\text{rot}~G_hu_I\|_{0}\lesssim h^\alpha\|u\|_{2+\alpha}.
\end{equation}
Here $u_I$ is the linear interpolation of $u$.
\end{lmm}

\begin{proof}
By the triangle inequality and the inverse inequality, we have, $\forall~\mathbf{w}_h\in V_h\times V_h$,
\begin{eqnarray*}
 \|D^2 u-DG_hu_I\|_{0}&\leq& \|D^2 u-D\mathbf{w}_h\|_{0}+\|D\mathbf{w}_h-DG_hu_I\|_{0}\\
  &\lesssim&\|D^2 u-D\mathbf{w}_h\|_{0}+h^{-1}\|\mathbf{w}_h-G_hu_I\|_{0}.
\end{eqnarray*}
Setting $\mathbf{w}_h=(\nabla u)_{I}$, then $\|D^2 u-D\mathbf{w}_h\|_{0}\lesssim h^\alpha\|u\|_{2+\alpha}$. Using $\eqref{Consistency}$ and standard estimates of linear interpolation, we derive that
\begin{eqnarray*}
h^{-1}\|(\nabla u)_{I}-G_hu_I\|_{0}&\leq& h^{-1}(\|(\nabla u)_{I}-\nabla u\|_{0}+\|\nabla u-G_hu_I\|_{0})\\
  &\lesssim&h^\alpha\|u\|_{2+\alpha}.
\end{eqnarray*}
Hence,
\begin{equation}\nonumber
   \|D^2 u-DG_hu_I\|_{0}\lesssim h^\alpha\|u\|_{2+\alpha}.
\end{equation}
Notice that,
\begin{equation*}
  \|\text{rot}~G_hu_I\|_0=\|\text{rot}~G_hu_I-\text{rot}~\nabla u\|_0\leq \|D^2 u-DG_hu_I\|_{0},
\end{equation*}
we get $\eqref{inequalityuI}$ immediately.
\end{proof}

\begin{thm}\label{thm:H2 error}
Let $u$ and $u_h$ be solutions of $\eqref{eqn:para}$ and $\eqref{discrete weak formulation}$, respectively. If $u\in H^{2+\alpha}(\Omega)$, then for any general grids, there holds
\begin{equation}\label{3.8}
 \|u_h-u_I\|_{a,1} \lesssim  h^\alpha\|u\|_{2+\alpha}.
\end{equation}
Consequently, we have
\begin{equation}\label{H2 estimate}
  \|D^2u-DG_hu_h\|_0 \lesssim   h^\alpha\|u\|_{2+\alpha}.
\end{equation}
\end{thm}

\begin{proof}
As $u$ solves $(1.1)$ strongly in $L^2(\Omega)$, it holds that
\begin{eqnarray*}
  \|u_h-u_I\|_{a,1}^2&=&a_{1,h}(u_h-u_I,u_h-u_I)\\
  &=&a_{1,h}(u_h,u_h-u_I)-a_{1,h}(u_I,u_h-u_I)\\
  &=&(f,A:DG_h(u_h-u_I))-a_{1,h}(u_I,u_h-u_I)\\
  &=&(A:(D^2u-DG_hu_I),A:DG_h(u_h-u_I))-(\text{rot}G_h u_I,\text{rot}G_h(u_h-u_I)).
\end{eqnarray*}
Using the Cauchy-Schwarz inequality and the triangle inequality, it yields that
\begin{eqnarray*}
  \|u_h-u_I\|_{a,1}^2&\leq&\|A\|_\infty\|D^2u-DG_hu_I\|_0\cdot\|A:DG_h(u_h-u_I)\|_0\\
  &&+\|\text{rot}G_hu_I\|_0\cdot \|\text{rot}G_h(u_h-u_I)\|\\
  &\lesssim&\|u_h-u_I\|_{a,1}\cdot(\|D^2u-DG_hu_I\|_0+\|\text{rot}G_hu_I\|_0).
\end{eqnarray*}
Dividing $\|u_h-u_I\|_a$ on both sides, we have
\begin{equation}\label{uhuIerra}
  \|u_h-u_I\|_{a,1}\lesssim\|D^2u-DG_hu_I\|_0+\|\text{rot}G_hu_I\|_0.
\end{equation}
Plugging $\eqref{inequalityuhuI}$ and $\eqref{inequalityuI}$ into $\eqref{uhuIerra}$, we obtain the estimate $\eqref{3.8}$.

Combining $\eqref{inequality6}$ and $\eqref{3.8}$, we derive that
\begin{equation}\label{DuhuIerra}
   \|DG_h(u_h-u_I)\|_0\lesssim h^\alpha\|u\|_{2+\alpha}.
\end{equation}
The estimate $\eqref{H2 estimate}$ is a direct consequence of  $\eqref{inequalityuhuI}$ and $\eqref{DuhuIerra}$.
\end{proof}

\begin{remark}By \eqref{DuhuIerra}, we have
\[
   \|{\rm rot}~ G_h(u_h-u_I)\|_0\lesssim h^\alpha\|u\|_{2+\alpha}.
\]
Combining the above estimate with \eqref{inequalityuI}, we obtain that
\begin{equation}\label{rotuh}
   \|{\rm rot}~G_h u_h\|_0\lesssim h^\alpha\|u\|_{2+\alpha}.
\end{equation}
\end{remark}

\begin{remark}
We observe that the proof of convergence and stability analysis for the scheme \eqref{discrete weak formulation} is not applicable to the scheme \eqref{discrete weak formulation 2}. Nonetheless, numerical results confirm the robustness and the optimal convergence of the scheme \eqref{discrete weak formulation 2}.
\end{remark}

\subsection{Error Estimates in $H^1$ and $L^2$ Norms}
In this subsection, we establish error bounds in a special case of $A=\alpha I$, where $\alpha$ is a constant. We first apply the Aubin-Nitsche technique to estimate the $H^1$ error. To this end, we introduce the following auxiliary problem. For $\phi \in H^1_0(\Omega)$, let $\partial \Omega$ be $C^3$ and $U_\phi\in H_0^1 (\Omega)$ be a weak solution of the following equation:
\begin{eqnarray*}
\left \{
\begin{array}{lll}
-\alpha^2 \Delta u=\phi \quad &\mbox{in} \quad~\Omega,\\
\quad~~~~~u=0  \quad &\mbox{on} \quad \partial\Omega.
\end{array}
\right.
\end{eqnarray*}
From [15, 6.3, Theorem 5], we have $U_\phi\in H^3(\Omega)$ and
\begin{equation}\label{priori estimate}
  \|U_\phi\|_3\lesssim\parallel\phi\|_1.
\end{equation}
By applying Green's formula, we derive that
\begin{eqnarray}\label{eqn:paranew}
\begin{array}{lll}
(\alpha^2 \Delta U_\phi, {\rm div} {\bf v})=({\bf v},\nabla \phi),\quad \forall~{\bf v}\in (H^1(\Omega))^2.
\end{array}
\end{eqnarray}

The following estimates hold.

\begin{thm}\label{thm:H1 error}
 Let $u$ and $u_h$ be solutions of $\eqref{eqn:para}$ and $\eqref{discrete weak formulation}$, respectively. If $u\in H^3(\Omega)$, then there holds
\begin{equation}\label{3.9}
  \|G_hu_h-G_hu_I\|_0 \lesssim h^2\|u\|_3,
\end{equation}
consequently,
\begin{equation}\label{H1 estimate}
  \|\nabla u-G_hu_h\|_0 \lesssim   h^2\|u\|_3.
\end{equation}
\end{thm}
\begin{proof}
For any $\phi\in H^1_0$, we have
\begin{eqnarray}\nonumber
  & & (G_h(u_h-u_I),\nabla \phi) \\ \nonumber
  &=&(\text{div}G_h(u_h-u_I),\alpha^2\Delta U_\phi) \\ \nonumber
                            &=&(\alpha\text{div}G_hu_h-\alpha\Delta u,\alpha\Delta U_\phi-\alpha\text{div}G_h(U_\phi)_I)+\\  \nonumber
                            &&(\alpha^2\Delta U_\phi,\Delta u-\text{div}G_hu_I)+(\alpha\text{div}G_hu_h-\alpha\Delta u,\alpha\text{div}G_h(U_\phi)_I)\\  \nonumber
                            &=&(\alpha\text{div}G_hu_h,\alpha\text{div}G_h(U_\phi)_I)-(\alpha\Delta u,\alpha\text{div}G_h(U_\phi)_I)+ \\ \nonumber
                            &&(\alpha\text{div}G_hu_h,\alpha\text{div}G_h(U_\phi)_I)-(f,\alpha\text{div}G_h(U_\phi)_I)\\ \nonumber
                            &=&(\alpha(\text{div}G_hu_h-\Delta u),\alpha(\Delta U_\phi-\text{div}G_h(U_\phi)_I))-(\nabla \phi,\nabla u-G_hu_I)+ \\ \nonumber
                            &&(\alpha\text{div}G_hu_h,\alpha\text{div}G_h(U_\phi)_I)-(\alpha\text{div}G_hu_h,\alpha\text{div}G_h(U_\phi)_I)-(\text{rot} G_hu_h,\text{rot} G_h(U_\phi)_I) \\ \nonumber
                            &=&(\alpha(\text{div}G_hu_h-\Delta u),\alpha(\Delta U_\phi-\text{div}G_h(U_\phi)_I))-(\nabla \phi,\nabla u-G_hu_I)-(\text{rot} G_hu_h,\text{rot} G_h(U_\phi)_I)\\ \label{error equation}
                            &\lesssim& h^2\|u\|_3(\|U_\phi\|_3+\|\phi\|_1)
\end{eqnarray}
where we have used the estimates \eqref{Consistency}, \eqref{inequalityuhuI}, \eqref{H2 estimate}, and \eqref{rotuh} in the last inequality.
Combining $\eqref{priori estimate}$ and $\eqref{error equation}$, we derive the desired estimation $\eqref{3.9}$.
The estimate $\eqref{H1 estimate}$ is a direct consequence of $\eqref{Consistency}$ and $\eqref{3.9}$.
\end{proof}
\begin{remark}
For the case of $A=\text{diag}(\alpha_1,\alpha_2)$, where $\alpha_1$ and $\alpha_2$ are constants with same sign, we have
\begin{equation*}
  A:D^2u=\alpha_1 u_{xx}+\alpha_2 u_{yy}.
\end{equation*}
By using variable substitution $x=\sqrt{\frac{\alpha_1}{\alpha_2}}s,y=t$, we obtain that
\begin{equation*}
  u_{xx}=\frac{\alpha_2}{\alpha_1}u_{ss},u_{yy}=u_{tt},
\end{equation*}
therefore $A:D^2u=\alpha_2(u_{ss}+u_{tt})$, indicating that the problem can be reduced to the situation $A=\alpha_2 I$.
\end{remark}

\begin{thm}\label{thm:L2 error}
Let $u$ and $u_h$ be solutions of $\eqref{eqn:para}$ and $\eqref{discrete weak formulation}$, respectively. If the mesh is sufficiently regular such that
$\eqref{Poincare inequality}$ holds and $u\in H^3(\Omega)$, then there holds
\begin{equation}\label{3.10}
  \|u_h-u\|_0 \lesssim h^2\|u\|_3,
\end{equation}
\end{thm}
\begin{proof}
By discrete Poincar\'{e} inequality $\eqref{Poincare inequality}$, we have
\begin{equation*}
  \|u_h-u_I\|_0\lesssim \|G_h(u_h-u_I)\|_0\lesssim h^2\|u\|_3.
\end{equation*}
Applying triangle inequality, we obtain
\begin{equation*}
  \|u-u_h\|_0\leq \|u-u_I\|_0+\|u_I-u_h\|_0\lesssim h^2\|u\|_3.
\end{equation*}
This completes the proof.
\end{proof}

\begin{remark}
The $H^1$ and $L^2$ error estimates are difficult for non-divergence elliptic equations in general, for which we have not yet found a theoretical proof in this article. Nonetheless, the $\mathcal{O}(h^2)$ order of convergence of GRBL and HRBL finite element schemes in $L^2$ norm can be confirmed by numerical experiments, even for problems with non-smooth and discontinuous coefficients.
\end{remark}

\section{Application to the Monge-Amp\`{e}re Equation}

In this section, we apply the recovered linear element method to solve the fully nonlinear Monge-Amp\`{e}re equation
\begin{eqnarray}\label{eqn:para6}
\left \{
\begin{array}{lll}
\text{det}(D^2u)&=&f \quad \mbox{in} \quad~\Omega,\\
\quad \quad \quad u&=&g  \quad \mbox{on} \quad \partial\Omega,
\end{array}
\right.
\end{eqnarray}
where $\Omega\subseteq \mathbb{R}^2$, and $D^2u$ is the Hessian of the function $u$. The Monge-Amp\`{e}re equation arises naturally from differential geometry and has widely applications in applied science such as mass transportation meteorology and geostrophic fluid dynamics.

If $f>0$, $\Omega$ and $u$ are convex, and $D^2u$ is positive definite, then problem \eqref{eqn:para6} admits a unique solution. Numerical approximation of the Monge-Amp\`{e}re equation is very challenging. Some numerical schemes in finite difference methods and/or FEMs have been designed for fully nonlinear equations in recent years,
see, e.g., \cite{Benamou,Brenner,Chen,Froese2,Froese3,Kawecki20192,Kawecki20193,Kawecki20194,Kawecki20195} and the references therein.

In this paper, we first use the efficient Newton's technology to linearize the Monge-Amp\`{e}re equation. Given $u_0\in V$, let $\{u_k\}_{k=1}^{\infty}\in V$ be a sequence, such that
\begin{eqnarray}\label{eqn:para5}
    \left \{
    \begin{array}{rll}
        \text{cof}(D^2u^{k-1}):D^2u^{k} & = f+\text{det}(D^2 u^{k-1}) & \mbox{in} \quad~\Omega,\\
        u^{k} & = g & \mbox{on} \quad \partial\Omega,
    \end{array}
    \right.
\end{eqnarray}
where the cofactor matrix of the Hessian $D^2u$ is defined as follows:
$$
 \begin{array}{ll}
\text{cof}~(D^2u)=\begin{pmatrix}
                u_{yy}  & -u_{yx}  \\
               -u_{xy}  &  u_{yy}\\
             \end{pmatrix}.
\end{array}
$$
For more details about the Newton's method, we refer to \cite{Lakkisand,Lakkis,Loeper}. It has been proved in \cite{Lakkis,Loeper,Pryer} that each iteration $u^k$ in the continuous Newton's scheme \eqref{eqn:para5} will be convex provided that the initial guess is strictly convex. From \cite{Dean,Lakkis}, a reasonable initial guess data is the solution of
\begin{eqnarray}\label{eqn:para7}
\left \{
\begin{array}{lll}
\Delta u^0&=&2\sqrt{f} \quad \mbox{in} \quad~\Omega,\\
\quad u^0&=&g  \quad~~~~ \mbox{on} \quad \partial\Omega.
\end{array}
\right.
\end{eqnarray}

Clearly, problem \eqref{eqn:para5} is an elliptic equation in non-divergence form. Define
$$V_h^{g}=\left \{v_h\in V_h: v_h|_{\partial \Omega}=g\right\}.$$
The gradient recovery linear element method for solving \eqref{eqn:para5} is to find $\{u_h^k\}_{k=1}^{\infty}\in V_h^{g}$ such that, $\forall v_h\in V_h$,
\begin{eqnarray}
    \nonumber
    && (\text{cof}~(DG_hu^{k-1}_h):DG_hu_h^k,\text{cof}~(DG_hu^{k-1}_h):DG_hv_h)+\sigma (\text{rot}(G_hu_h^k),\text{rot}(G_hv_h)) \\
    \label{eqn:weakform2}
    &=& (f+\text{det}(DG_h u^{k-1}_h),\text{cof}~(DG_hu^{k-1}_h):DG_hv_h),
\end{eqnarray}
where $D_h^2u_h$ is the weak Hessian of the linear function $u_h$ and $\sigma>0$ is a penalty parameter.

We will exhibit the robustness and convergence of the proposed algorithm in Section 6.2 by using numerical examples. We observe that the solution of the scheme \eqref{eqn:weakform2} can convergent to a convex solution.

\section{Numerical Experiments}
In this section, we present numerical results for some representative examples to confirm our theoretical findings. In all examples, uniform meshes are used. We apply PPR \eqref{gradint recovery formula} and \eqref{equ:hessian2} to generate $G_hu_h$ and $H_hu_h$ for HRBL and GRBL FEMs, respectively. We shall examine several numerical errors, which will be denoted in the following notations:
\begin{eqnarray*}
 L^2~\text{norm}&:& \|e \|_0 = \|u-u_h\|_{0,\Omega},\\
 H^1~\text{seminorm}&:& ~|e |_1=|u-u_h|_{1,\Omega}, \\
 \text{Recovered}~H^1~\text{seminorm}&:&~|e |_{1,r} = \|\nabla u-G_hu_h\|_{0,\Omega},\\
  H^2~\text{seminorm}&:&~|e |_2=\|D^2 u-D_h^2u_h\|_{0,\Omega}.
\end{eqnarray*}
where $D_h^2u_h=DG_hu_h$ for the scheme \eqref{discrete weak formulation} and $D_h^2u_h=H_hu_h$ for the scheme \eqref{discrete weak formulation 2}.

\subsection{Numerical Experiments of Non-divergence Form Elliptic Equations}

We first consider four examples of second-order linear elliptic PDEs, including examples with non-smooth and/or discontinuous coefficients over convex domains (cf. Figure~\ref{fig:msh}) or $L$-shaped domain.
\begin{figure}[!h]
    \centering
    {\includegraphics[width=0.35\textwidth]{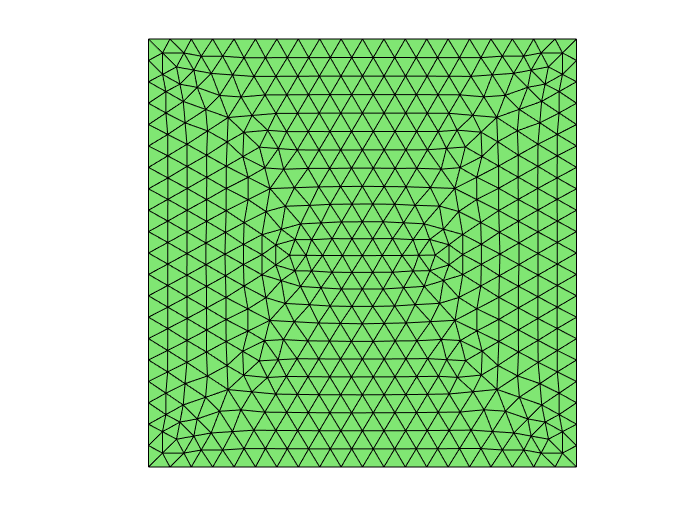}}
    {\includegraphics[width=0.35\textwidth]{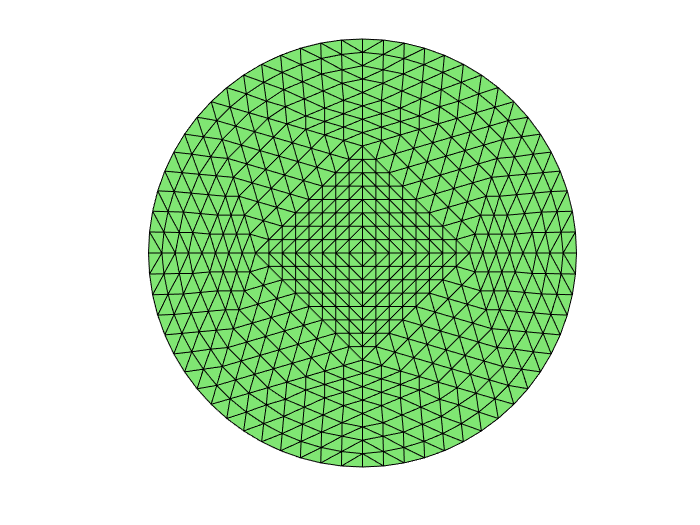}}
    \caption{(a) A nonuniform mesh (b) A mesh on circular area }
    \label{fig:msh}
\end{figure}

\begin{ex}\label{ex:problem 1} A problem with non-smooth coefficients.
\end{ex}
We consider numerical approximation for the problem $\eqref{eqn:para}$ with exact solution $u=\sin x \sin y$. The coefficient function is
\begin{equation*}
    A=\begin{pmatrix}
        1+ |x|           & 0.5 |xy|^{1/3} \\
        0.5 |xy|^{1/3}  & 1+ |y|\\
    \end{pmatrix},
\end{equation*}
which is continuous. This elliptic problem cannot be written in a divergence form, because the two off diagonal entries have a singularity at the origin.

\begin{table}[!h]
\centering
\small{\caption{\emph{Example $\ref{ex:problem 1}$ -- Numerical results of the scheme \eqref{discrete weak formulation} on uniform meshes}}}
\label{Tab:Eg1a}
\begin{tabular}{lllllllllll}
\hline
 $~1/h $ &~~$\|e\|_0$   &order  &~~$|e|_1$  &order  &~~$|e|_{1,r}$&order  &~~$|e|_2$  &order\\
\hline
 $~~16$ & 2.19E-3 &~      &1.09E-1  &$~$     &1.04E-2   &$~$    &1.15E-2   &$~$\\
 $~~32$ & 5.37E-4 &~2.03  &5.44E-1  &$~1.01$ &2.56E-3   &$~2.01$&5.63E-2   &$~1.03$\\
 $~~64$ & 1.33E-4 &~2.01  &2.72E-2  &$~1.00$ &6.35E-4   &$~2.01$&2.79E-2   &$~1.01$\\
 $~128$ & 3.31E-5 &~2.00  &1.36E-2  &$~1.00$ &1.58E-4   &$~2.01$&1.39E-2   &$~1.00$\\
 $~256$ & 8.26E-6 &~2.00  &6.79E-3  &$~1.00$ &3.94E-5   &$~2.00$&6.96E-3   &$~1.00$\\
 $~512$ & 2.06E-6 &~2.00  &3.39E-3  &$~1.00$ &9.82E-6   &$~2.00$&3.48E-3   &$~1.00$\\
\hline
\end{tabular}
\end{table}

\begin{table}[!h]
\centering
\small{\caption{\emph{Example $\ref{ex:problem 1}$ -- Numerical results of the scheme \eqref{discrete weak formulation} on nonuniform meshes}}}
\label{Tab:Eg1b}
\begin{tabular}{lllllllllll}
\hline
 $~~n $ &~~$\|e\|_0$   &order  &~~$|e|_1$  &order  &~~$|e|_{1,r}$&order  &~~$|e|_2$  &order\\
\hline
 $~~2$ & 1.00E-3 &~      &6.17E-2  &$~$     &5.28E-2   &$~$    &6.80E-1   &$~$\\
 $~~3$ & 2.46E-4 &~2.02  &3.08E-2  &$~1.00$ &1.28E-3   &$~2.04$&3.29E-1   &$~1.05$\\
 $~~4$ & 6.07E-5 &~2.02  &1.54E-2  &$~1.00$ &3.11E-4   &$~2.04$&1.61E-1   &$~1.03$\\
 $~~5$ & 1.50E-5 &~2.02  &7.69E-3  &$~1.00$ &7.68E-5   &$~2.02$&8.00E-2   &$~1.01$\\
\hline
\end{tabular}
\end{table}

\begin{table}[!h]
\centering
\small{\caption{\emph{Example $\ref{ex:problem 1}$ -- Numerical results of the scheme \eqref{discrete weak formulation 2} on uniform meshes}}}
\label{Tab:Eg1c}
\begin{tabular}{lllllllllll}
\hline
 $~1/h $ &~~$\|e\|_0$   &order  &~~$|e|_1$  &order  &~~$|e|_{1,r}$&order  &~~$|e|_2$  &order\\
\hline
 $~~16$ & 2.20E-3 &~      &1.09E-1  &$~$     &1.05E-2   &$~$    &6.56E-2   &$~$\\
 $~~32$ & 5.47E-4 &~2.03  &5.43E-1  &$~1.01$ &2.60E-3   &$~2.01$&2.34E-2   &$~1.49$\\
 $~~64$ & 1.36E-4 &~2.01  &2.72E-2  &$~1.00$ &6.43E-4   &$~2.01$&8.33E-3   &$~1.49$\\
 $~128$ & 3.39E-5 &~2.00  &1.36E-2  &$~1.00$ &1.60E-4   &$~2.01$&2.95E-3   &$~1.49$\\
 $~256$ & 8.48E-6 &~2.00  &6.79E-3  &$~1.00$ &3.98E-5   &$~2.00$&1.05E-3   &$~1.50$\\
 $~512$ & 2.12E-6 &~2.00  &3.39E-3  &$~1.00$ &9.82E-6   &$~2.00$&3.74E-4   &$~1.50$\\
\hline
\end{tabular}
\end{table}

\begin{table}[!h]
\centering
\small{\caption{\emph{Example $\ref{ex:problem 1}$ -- Numerical results of the scheme \eqref{discrete weak formulation 2} on nonuniform meshes}}}
\label{Tab:Eg1d}
\begin{tabular}{lllllllllll}
\hline
 $~~n $ &~~$\|e\|_0$   &order  &~~$|e|_1$  &order  &~~$|e|_{1,r}$&order  &~~$|e|_2$  &order\\
\hline
 $~~2$ & 9.63E-3 &~      &6.15E-2  &$~$     &5.37E-2   &$~$    &2.65E-2   &$~$\\
 $~~3$ & 2.42E-4 &~2.02  &3.08E-2  &$~1.00$ &1.29E-3   &$~2.04$&9.37E-2   &$~1.50$\\
 $~~4$ & 6.04E-5 &~2.02  &1.54E-2  &$~1.00$ &3.15E-4   &$~2.04$&3.32E-2   &$~1.50$\\
 $~~5$ & 1.50E-5 &~2.02  &7.69E-3  &$~1.00$ &7.71E-5   &$~2.02$&1.18E-2   &$~1.50$\\
\hline
\end{tabular}
\end{table}

\begin{table}[!h]
\centering
\small{\caption{\emph{Example $\ref{ex:problem 1}$ -- Numerical results  of the scheme \eqref{discrete weak formulation} on a $L$-Shaped domain}}}
\label{Tab:Eg1e}
\begin{tabular}{lllllllllll}
\hline
$~1/h $ &~~$\|e\|_0$   &order  &~~$|e|_1$  &order  &~~$|e|_{1,r}$&order  &~~$|e|_2$  &order\\
\hline
 $~~16$ & 2.33E-3 &~      &9.79E-1  &$~$     &8.53E-3   &$~$    &9.37E-2   &$~$\\
 $~~32$ & 5.76E-4 &~2.02  &4.88E-2  &$~1.01$ &2.11E-3   &$~2.01$&4.55E-2   &$~1.04$\\
 $~~64$ & 1.42E-4 &~2.01  &2.44E-2  &$~1.00$ &5.23E-4   &$~2.01$&2.25E-2   &$~1.02$\\
 $~128$ & 3.55E-5 &~2.00  &1.22E-2  &$~1.00$ &1.30E-4   &$~2.01$&1.12E-2   &$~1.01$\\
 $~256$ & 8.85E-6 &~2.00  &6.09E-3  &$~1.00$ &3.23E-5   &$~2.01$&5.59E-3   &$~1.00$\\
 $~512$ & 2.23E-6 &~2.00  &3.05E-3  &$~1.00$ &8.05E-6   &$~2.00$&2.79E-3   &$~1.00$\\
\hline
\end{tabular}
\end{table}

\begin{table}[!h]
\centering
\small{\caption{\emph{Example $\ref{ex:problem 1}$ -- Numerical results  of the scheme \eqref{discrete weak formulation 2} on a $L$-Shaped domain}}}
\label{Tab:Eg1f}
\begin{tabular}{lllllllllll}
\hline
$~1/h $ &~~$\|e\|_0$   &order  &~~$|e|_1$  &order  &~~$|e|_{1,r}$&order  &~~$|e|_2$  &order\\
\hline
 $~~16$ & 2.40E-3 &~      &9.75E-1  &$~$     &8.77E-3   &$~$    &5.49E-2   &$~$\\
 $~~32$ & 5.84E-4 &~1.98  &4.87E-2  &$~1.01$ &2.16E-3   &$~1.95$&1.98E-2   &$~1.11$\\
 $~~64$ & 1.44E-4 &~1.98  &2.44E-2  &$~1.00$ &5.30E-4   &$~2.01$&7.06E-3   &$~1.07$\\
 $~128$ & 3.58E-5 &~1.98  &1.22E-2  &$~1.00$ &1.31E-4   &$~2.01$&2.51E-3   &$~1.03$\\
 $~256$ & 8.95E-6 &~1.99  &6.09E-3  &$~1.00$ &3.24E-5   &$~2.00$&8.88E-4   &$~1.01$\\
\hline
\end{tabular}
\end{table}

\begin{table}[!h]
\centering
\small{\caption{\emph{Example $\ref{ex:problem 1}$ -- Numerical results  of the scheme \eqref{discrete weak formulation} on a circular area}}}
\label{Tab:Eg1g}
\begin{tabular}{lllllllllll}
\hline
$~1/h $ &~~$\|e\|_0$   &order  &~~$|e|_1$  &order  &~~$|e|_{1,r}$&order  &~~$|e|_2$  &order\\
\hline
 $~~16$ & 2.46E-3 &~      &1.07E-1  &$~$     &1.38E-3   &$~$    &1.23E-1   &$~$\\
 $~~32$ & 6.26E-4 &~2.02  &5.34E-2  &$~1.01$ &3.57E-3   &$~2.01$&5.70E-2   &$~1.11$\\
 $~~64$ & 1.59E-4 &~2.01  &2.67E-2  &$~1.00$ &8.88E-4   &$~2.01$&2.72E-2   &$~1.07$\\
 $~128$ & 4.02E-5 &~2.00  &1.33E-2  &$~1.00$ &2.21E-4   &$~2.01$&1.33E-2   &$~1.03$\\
 $~256$ & 1.01E-5 &~2.00  &6.65E-3  &$~1.00$ &5.51E-5   &$~2.01$&6.60E-3   &$~1.01$\\
\hline
\end{tabular}
\end{table}

\begin{table}[!h]
\centering
\small{\caption{\emph{Example $\ref{ex:problem 1}$ -- Numerical results  of the scheme \eqref{discrete weak formulation 2} on a circular area}}}
\label{Tab:Eg1r}
\begin{tabular}{lllllllllll}
\hline
$~1/h $ &~~$\|e\|_0$   &order  &~~$|e|_1$  &order  &~~$|e|_{1,r}$&order  &~~$|e|_2$  &order\\
\hline
 $~~16$ & 2.35E-3 &~      &1.07E-1  &$~$     &1.29E-3   &$~$    &7.12E-1   &$~$\\
 $~~32$ & 5.89E-4 &~2.00  &5.32E-2  &$~1.01$ &3.42E-3   &$~1.91$&2.67E-2   &$~1.42$\\
 $~~64$ & 1.51E-4 &~1.96  &2.66E-2  &$~1.00$ &8.71E-4   &$~1.97$&9.70E-2   &$~1.46$\\
 $~128$ & 3.84E-5 &~1.97  &1.33E-2  &$~1.00$ &2.19E-4   &$~1.99$&3.47E-2   &$~1.48$\\
 $~256$ & 9.71E-5 &~1.98  &6.66E-3  &$~1.00$ &5.50E-5   &$~2.00$&1.23E-3   &$~1.49$\\
\hline
\end{tabular}
\end{table}

Numerical results of the recovered linear finite element approximation for problem $\eqref{eqn:para}$ on a square $\Omega=(-1,1)^2$, an $L$-shaped domain $\Omega=(-1,1)^2 \setminus (0,1)^2$ and a circular domain $\Omega=\{(x,y)|x^2+y^2\leq 1\}$ are collected in Tables~\ref{Tab:Eg1a}-6.8. For the GRBL FEM \eqref{discrete weak formulation} on both uniform and unstructured meshes (see Figure~\ref{fig:msh}), the convergence orders of numerical errors in $L^2$- and $H^1$- norms are $\mathcal{O}(h^2)$ and $\mathcal{O}(h)$, respectively, which are both optimal. Superconvergence phenomenon is also observed. In particular, $DG_hu_h$ and $G_hu_h$ converge to the $D^2u$ and $\nabla u$ with convergence rates $\mathcal{O}(h)$ and $\mathcal{O}(h^2)$, respectively. As for the HRBL FEM \eqref{discrete weak formulation 2}, the convergence rates of numerical errors in $L^2$-, $H^1$-, and recovered gradient norms are the same as those from the scheme \eqref{discrete weak formulation}; but the convergence rate in $H^2$-seminorm is $\mathcal{O} (h^{1.5})$, which is half order higher than that of the first scheme \eqref{discrete weak formulation}. It can also be concluded from Tables~6.7-6.8 that our proposed methods maintain optimal convergence rates over a circular domain. Moreover, numerical we observe that the convergence order under $H^2$-seminorm error obtained by the scheme (3.6) is 1.5 for the convex domain, but it reduces to 1.0 for non-convex domains.

\begin{ex}\label{ex:problem 2} A problem with discontinuous coefficients.
\end{ex}
In this example, we consider the test problem in \cite{Gallistl,Smears,Wang}. The coefficient reads
\begin{equation*}
    A=\begin{pmatrix}
        2 & xy/|xy| \\
        xy/|xy| & 2 \\
    \end{pmatrix},
\end{equation*}
and the function $f$ is chosen such that the exact solution of $\eqref{eqn:para}$ is
$$u(x,y)=xy(1-\exp(1-|x|))(1-\exp(1-|y|)).$$
Note that the coefficient is discontinuous across the $x$- and $y$-axis. It is straightforward to verify that example $\ref{ex:problem 2}$ satisfies Cordes condition with $\epsilon=3/5$. Tables~6.9-6.11 demonstrate the performance of the GRBL FEM for the test problem over a square domain, a $L$-shaped domain, and a circular domain, respectively, as specified in Example~\ref{ex:problem 1}. Numerical results illustrate that the convergence rates over these domains are all optimal for problems with discontinuous coefficients, which are $\mathcal{O}(h)$ for the approximation of the Hessian, $\mathcal{O}(h)$ for the approximation of the gradient in $H^1$ norm, and $\mathcal{O}(h^2)$ for the approximation of exact solution in $L^2$ norm. It worths mentioning that the convergence order of the WG method in $L^2$ norm \cite{Wang} and the internal penalty method in $L^2$ norm \cite{Mu} are both $\mathcal{O}(h^2)$ when quadratic elements are employed.

\begin{table}[!h]
\centering
\small{\caption{\emph{Example $\ref{ex:problem 2}$ -- Numerical results of the scheme \eqref{discrete weak formulation} on a square domain}}}
\label{Tab:Eg2a}
\begin{tabular}{lllllllllll}
\hline
 $~1/h $ &~~$\|e\|_0$   &order  &~~$|e|_1$  &order  &~~$|e|_{1,r}$&order  &~~$|e|_2$  &order\\
\hline
 $~~16$ & 7.71E-3 &~      &1.94E-1  &$~$     &2.25E-2   &$~$    &6.41E-1   &$~$\\
 $~~32$ & 1.90E-3 &~2.02  &9.33E-2  &$~1.05$ &5.73E-3   &$~1.97$&3.14E-1   &$~1.03$\\
 $~~64$ & 4.81E-4 &~1.98  &4.66E-2  &$~1.00$ &1.46E-3   &$~1.97$&1.55E-1   &$~1.02$\\
 $~128$ & 1.21E-4 &~1.99  &2.32E-2  &$~1.01$ &3.69E-4   &$~1.98$&7.72E-2   &$~1.01$\\
 $~256$ & 3.02E-5 &~2.00  &1.15E-2  &$~1.01$ &9.31E-5   &$~1.99$&3.83E-2   &$~1.01$\\
 $~512$ & 7.55E-6 &~2.00  &5.07E-3  &$~1.00$ &2.33E-5   &$~1.99$&1.90E-2   &$~1.01$\\
\hline
\end{tabular}
\end{table}

\begin{table}[!h]
\centering
\small{\caption{\emph{Example $\ref{ex:problem 2}$ -- Numerical results of the scheme \eqref{discrete weak formulation} on a $L$-Shaped domain}}}
\label{Tab:Eg2b}
\begin{tabular}{lllllllllll}
\hline
$~1/h $ &~~$\|e\|_0$   &order  &~~$|e|_1$  &order  &~~$|e|_{1,r}$&order  &~~$|e|_2$  &order\\
\hline
 $~~16$ & 8.03E-3 &~      &1.65E-1  &$~$     &2.75E-2   &$~$    &5.50E-1   &$~$\\
 $~~32$ & 2.44E-3 &~1.72  &7.49E-2  &$~1.14$ &8.09E-3   &$~1.77$&2.63E-1   &$~1.06$\\
 $~~64$ & 6.87E-4 &~1.83  &4.69E-2  &$~1.02$ &2.22E-3   &$~1.86$&1.29E-1   &$~1.03$\\
 $~128$ & 1.81E-4 &~1.93  &1.82E-2  &$~1.02$ &5.83E-4   &$~1.94$&6.38E-2   &$~1.01$\\
 $~256$ & 4.51E-5 &~2.00  &9.01E-3  &$~1.02$ &1.47E-4   &$~1.99$&3.19E-2   &$~1.00$\\
 $~512$ & 1.13E-5 &~2.00  &4.45E-3  &$~1.02$ &3.67E-5   &$~2.00$&1.60E-2   &$~1.00$\\
\hline
\end{tabular}
\end{table}

\begin{table}[!h]
\centering
\small{\caption{\emph{Example $\ref{ex:problem 2}$ -- Numerical results of the scheme \eqref{discrete weak formulation} on a circular area}}}
\label{Tab:Eg2c}
\begin{tabular}{lllllllllll}
\hline
$~1/h $ &~~$\|e\|_0$   &order  &~~$|e|_1$  &order  &~~$|e|_{1,r}$&order  &~~$|e|_2$  &order\\
\hline
 $~~16$ & 6.65E-3 &~      &1.74E-1  &$~$     &4.95E-2   &$~$    &8.32E-1   &$~$\\
 $~~32$ & 1.91E-3 &~1.80  &9.17E-2  &$~0.92$ &1.14E-2   &$~2.11$&4.13E-1   &$~1.01$\\
 $~~64$ & 4.86E-4 &~1.98  &4.58E-2  &$~1.00$ &2.73E-3   &$~2.07$&2.06E-1   &$~1.00$\\
 $~128$ & 1.18E-4 &~2.04  &2.23E-2  &$~1.04$ &6.73E-4   &$~2.02$&1.04E-1   &$~0.98$\\
 $~256$ & 2.89E-5 &~2.03  &1.09E-2  &$~1.03$ &1.69E-4   &$~1.99$&5.35E-2   &$~0.97$\\
\hline
\end{tabular}
\end{table}

\begin{ex}\label{ex:problem 3} A problem with a singular solution.
\end{ex}
In this example, we consider the problem suggested in \cite{Wang,Smears}. The test equation is given by
\begin{equation}\label{eqn:problem 3}
   \sum_{i,j=1}^2(\delta_{i,j}+\frac{x_ix_j}{|{\bf x}|^2})\partial_{ij}^2 u=f \quad \text{in}~\Omega,
\end{equation}
where $\delta_{i,j}$ is the Kronecker delta and ${\bf x} = (x_1,x_2)$. For $\alpha>1$, it is straightforward to confirm that $u=|{\bf x}|^\alpha\in H^2(\Omega)$ satisfies $\eqref{eqn:problem 3}$ with $f=(2\alpha^2-\alpha)|{\bf x}|^{\alpha-2}$. In fact, the solution $u \in H^{1+\alpha-\tau}(\Omega)$ for arbitrarily small $\tau>0$. Moreover, the coefficient satisfies the Cordes condition with $\epsilon=4/5$. In the numerical experiments, we take $\alpha=1.6$ with problem $\eqref{eqn:problem 3}$ defined on two square domains $(0,1)^2$ and $(-1,1)^2$.

Results in Tables~6.12-6.15 are from the computational domain $\Omega=(0,1)^2$, for which the coefficient matrix is discontinuous at the origin. Numerical results suggest a convergence rate of $\mathcal{O}(h^{0.6})$ in the $H^2$ seminorm, which is consistent with the estimate \eqref{H2 estimate}. The convergence rates in $L^2$ norm and $H^1$ seminorm are of $\mathcal{O}(h^{2})$ and $\mathcal{O}(h)$, respectively. The recovered numerical gradient has a superconvergence order of $\mathcal{O}(h^{1.6})$. Tables~6.16 and 6.17 display the performance of the recovered linear finite element schemes for $\eqref{eqn:problem 3}$ on the domain $\Omega=(-1,1)^2$. Due to the discontinuity of the coefficient matrix at the origin, the convergence rates of numerical results are reduced. In particular, the numerical results suggest a convergence rate of $\mathcal{O}(h^{0.6})$ in the $H^2$ seminorm. The convergence rates in both recovered gradient norm and $L^2$ norm are of $\mathcal{O}(h^{1.1})$, which is consistent with the numerical results reported in \cite{Wang}.

\begin{table}[!h]
\centering
\small{\caption{\emph{Example~\ref{ex:problem 3} -- Numerical results of the scheme \eqref{discrete weak formulation} on uniform meshes}}}
\label{Tab:Eg3a}
\begin{tabular}{lllllllllll}
\hline
 $~1/h $ &~~$\|e\|_0$   &order  &~~$|e|_1$  &order  &~~$|e|_{1,r}$&order  &~~$|e|_2$  &order\\
\hline
 $~~32$ & 3.15E-4 &1.93  &1.92E-2 &$1.07$ &1.45E-3 &$1.58$&1.33E-1 &$~0.60$\\
 $~~64$ & 8.02E-5 &1.98  &9.23E-3 &$1.05$ &4.81E-4 &$1.59$&8.76E-2 &$~0.60$\\
 $~128$ & 1.98E-5 &2.02  &4.52E-3 &$1.03$ &1.59E-4 &$1.60$&5.78E-2 &$~0.60$\\
 $~256$ & 4.69E-6 &2.07  &2.23E-3 &$1.02$ &5.26E-5 &$1.60$&3.82E-2 &$~0.60$\\
 $~512$ & 1.06E-6 &2.14  &1.11E-3 &$1.01$ &1.74E-5 &$1.60$&2.52E-2 &$~0.60$\\
\hline
\end{tabular}
\end{table}

\begin{table}[!h]
\centering
\small{\caption{\emph{Example~\ref{ex:problem 3} -- Numerical results of the scheme \eqref{discrete weak formulation} on nonuniform meshes}}}
\label{Tab:Eg3b}
\begin{tabular}{lllllllllll}
\hline
 $  n $ &~~$\|e\|_0$   &order  &~~$|e|_1$  &order  &~~$|e|_{1,r}$&order  &~~$|e|_2$  &order\\
\hline
 $2$ & 6.29E-4 &      &2.76E-2 &       &2.67E-3 &      &1.49E-1 &       \\
 $3$ & 1.70E-4 &1.89  &1.36E-2 &$1.02$ &9.22E-4 &$1.53$&1.03E-1 &$~0.53$\\
 $4$ & 4.50E-5 &1.92  &6.59E-3 &$1.05$ &2.89E-4 &$1.67$&6.81E-2 &$~0.59$\\
 $5$ & 1.20E-5 &1.90  &3.21E-3 &$1.04$ &9.43E-5 &$1.62$&4.43E-2 &$~0.62$\\
\hline
\end{tabular}
\end{table}

\begin{table}[!h]
\centering
\small{\caption{\emph{Example~\ref{ex:problem 3} -- Numerical results of the scheme \eqref{discrete weak formulation 2} on uniform meshes}}}
\label{Tab:Eg3c}
\begin{tabular}{lllllllllll}
\hline
 $~1/h $ &~~$\|e\|_0$   &order  &~~$|e|_1$  &order  &~~$|e|_{1,r}$&order  &~~$|e|_2$  &order\\
\hline
 $~~32$ & 2.63E-4 &      &1.75E-2 &       &1.41E-3 &      &1.10E-1 &       \\
 $~~64$ & 6.61E-5 &1.99  &8.74E-3 &$1.00$ &4.68E-4 &$1.59$&7.23E-2 &$~0.60$\\
 $~128$ & 1.58E-5 &2.06  &4.37E-3 &$1.00$ &1.56E-4 &$1.59$&4.77E-2 &$~0.60$\\
 $~256$ & 3.75E-6 &2.07  &2.18E-3 &$1.00$ &5.18E-5 &$1.59$&3.15E-2 &$~0.60$\\
 $~512$ & 8.83E-6 &2.08  &1.09E-3 &$1.01$ &1.72E-5 &$1.60$&2.08E-2 &$~0.60$\\
\hline
\end{tabular}
\end{table}

\begin{table}[!h]
\centering
\small{\caption{\emph{Example~\ref{ex:problem 3} -- Numerical results of the scheme \eqref{discrete weak formulation 2} on nonuniform meshes}}}
\label{Tab:Eg3d}
\begin{tabular}{lllllllllll}
\hline
 $  n $ &~~$\|e\|_0$   &order  &~~$|e|_1$  &order  &~~$|e|_{1,r}$&order  &~~$|e|_2$  &order\\
\hline
 $2$ & 5.39E-4 &      &2.43E-2 &       &2.53E-3 &      &1.27E-1 &       \\
 $3$ & 1.46E-4 &1.90  &1.21E-2 &$1.01$ &8.14E-4 &$1.63$&8.42E-1 &$~0.59$\\
 $4$ & 3.75E-5 &1.96  &6.00E-3 &$1.00$ &2.69E-4 &$1.60$&5.55E-2 &$~0.60$\\
 $5$ & 9.73E-5 &1.94  &3.00E-3 &$1.00$ &8.90E-5 &$1.60$&3.66E-2 &$~0.60$\\
\hline
\end{tabular}
\end{table}

\begin{table}[!h]
\centering
\small{\caption{\emph{Example~\ref{ex:problem 3} -- Numerical results of the scheme \eqref{discrete weak formulation}}}}
\label{Tab:Eg3e}
\begin{tabular}{lllllllllll}
\hline
$~1/h $ &~~$\|e\|_0$   &order  &~~$|e|_1$  &order  &~~$|e|_{1,r}$&order  &~~$|e|_2$  &order\\
\hline
 $~~32$ & 8.63E-3 &~     &8.52E-2 &$~   $ &2.21E-2 &$~   $&3.42E-1 &$~    $\\
 $~~64$ & 3.99E-3 &1.11  &4.27E-2 &$1.00$ &1.02E-3 &$1.11$&2.27E-1 &$~0.59$\\
 $~128$ & 1.85E-3 &1.10  &2.14E-2 &$1.00$ &4.75E-3 &$1.11$&1.51E-1 &$~0.59$\\
 $~256$ & 8.65E-4 &1.10  &1.07E-2 &$1.00$ &2.20E-3 &$1.11$&9.96E-2 &$~0.60$\\
 $~512$ & 4.03E-4 &1.10  &5.04E-3 &$1.00$ &1.02E-3 &$1.11$&6.58E-2 &$~0.60$\\
\hline
\end{tabular}
\end{table}

\begin{table}[!h]
\centering
\small{\caption{\emph{Example~\ref{ex:problem 3} -- Numerical results of the scheme \eqref{discrete weak formulation 2}}}}
\label{Tab:Eg3f}
\begin{tabular}{lllllllllll}
\hline
$~1/h $ &~~$\|e\|_0$   &order  &~~$|e|_1$  &order  &~~$|e|_{1,r}$&order  &~~$|e|_2$  &order\\
\hline
 $~~32$ & 6.29E-3 &~     &8.44E-2 &$~   $ &1.67E-2 &$~   $&2.05E-1 &$~    $\\
 $~~64$ & 2.79E-3 &1.17  &4.23E-2 &$1.00$ &7.36E-3 &$1.18$&1.35E-1 &$~0.60$\\
 $~128$ & 1.27E-3 &1.14  &2.12E-2 &$1.00$ &3.31E-3 &$1.15$&8.90E-2 &$~0.60$\\
 $~256$ & 5.85E-4 &1.12  &1.06E-2 &$1.00$ &1.52E-3 &$1.12$&5.87E-2 &$~0.60$\\
 $~512$ & 2.70E-4 &1.11  &5.03E-3 &$1.00$ &7.03E-4 &$1.11$&3.87E-2 &$~0.60$\\
\hline
\end{tabular}
\end{table}

\begin{ex}\label{ex:problem 4} A problem with degenerate coefficients.
\end{ex}
In this example, we consider the problem \eqref{eqn:para} with degenerate coefficients suggested in \cite{Feng}. The coefficient reads
\begin{equation*}
    A=\frac{16}{9}\begin{pmatrix}
        x^{2/3} & -x^{1/3}y^{1/3} \\
        -x^{1/3}y^{1/3} & y^{2/3} \\
    \end{pmatrix}.
\end{equation*}
The exact solution of this problem is set as $u=x^{4/3} - y^{4/3}$. We take $\Omega=(0,1)^2$. Note that $A:D^2 u =0$. Unlike the first three example problems, this problem is not uniformly elliptic as $\det(A) \equiv 0$ in $\Omega$. Therefore, the error estimates developed in this paper are not applicable. Nevertheless, numerical results by the GRBL FEM are presented in Table~6.18 and 6.19. The experiment illustrates that $u-u_h$ measured in the $L^2$ norm and the $H^1$ seminorm have convergence rates of $\mathcal{O}(h^{1.27})$ and $\mathcal{O}(h^{0.78})$, respectively, for the GRLEM; and are of $\mathcal{O}(h^{1.36})$ and $\mathcal{O}(h^{0.83})$, respectively, for the HRLEM. These rates are competitive to the numerical results reported in \cite{Feng}.

\begin{table}[!h]
\centering
\small{\caption{\emph{Example~\ref{ex:problem 4} -- Numerical results of the scheme \eqref{discrete weak formulation}}}}
\label{Tab:Eg4a}
\begin{tabular}{lllllllllll}
\hline
 $~1/h $ &~~$\|e\|_0$   &order  &~~$|e|_1$  &order  &~~$|e|_{1,r}$&order \\
\hline
 $~~32$ & 3.99E-4 &1.24  &2.58E-2 &$0.77$ &1.04E-2 &$0.88$\\
 $~~64$ & 1.72E-4 &1.22  &1.51E-2 &$0.77$ &5.69E-3 &$0.87$\\
 $~128$ & 7.27E-5 &1.24  &8.81E-3 &$0.78$ &3.10E-3 &$0.87$\\
 $~256$ & 3.00E-5 &1.27  &5.13E-3 &$0.78$ &1.70E-3 &$0.87$\\
 $~512$ & 1.24E-5 &1.27  &2.99E-3 &$0.78$ &9.31E-4 &$0.87$\\
\hline
\end{tabular}
\end{table}
\begin{table}[!h]
\centering
\small{\caption{\emph{Example~\ref{ex:problem 4} -- Numerical results of the scheme \eqref{discrete weak formulation 2}}}}
\label{Tab:Eg4b}
\begin{tabular}{lllllllllll}
\hline
 $~1/h $ &~~$\|e\|_0$   &order  &~~$|e|_1$  &order  &~~$|e|_{1,r}$&order \\
\hline
 $~~32$ & 6.00E-4 &      &3.39E-2 &       &2.10E-2 & \\
 $~~64$ & 2.40E-4 &1.32  &1.96E-2 &$0.79$ &1.28E-2 &$0.72$\\
 $~128$ & 9.41E-5 &1.35  &1.12E-2 &$0.81$ &7.46E-3 &$0.77$\\
 $~256$ & 3.67E-5 &1.36  &6.36E-3 &$0.82$ &4.28E-3 &$0.80$\\
 $~512$ & 1.43E-5 &1.36  &3.57E-3 &$0.83$ &2.40E-4 &$0.83$\\
\hline
\end{tabular}
\end{table}
\begin{ex}\label{ex:problem 15} A 3D problem with non-smooth coefficients.
\end{ex}
We extend our proposed method to solve a 3D problem on cuboid meshes with exact solution $u=\sin \pi x \sin \pi y\sin \pi z$. The coefficient function is
\begin{equation*}
    A=\begin{pmatrix}
        1+ |x|           & 0.5 |xy|^{1/3} & 0.5 |xz|^{1/3}\\
        0.5 |xy|^{1/3}   & 1+ |y|         & 0.5 |yz|^{1/3}\\
        0.5 |xz|^{1/3}   & 0.5 |yz|^{1/3} & 1+ |z|         \\
    \end{pmatrix},
\end{equation*}
which is continuous but has singularities at the origin. We take $\Omega=(-1,1)^3$.

Numerical results by the GRBL FEM and HRBL FEM are presented in Table~6.20 and 6.21. The experiment illustrates that $u-u_h$ measured in the $L^2$ norm and the $H^1$ seminorm have convergence rates of $\mathcal{O}(h^{2})$ and $\mathcal{O}(h^{1})$ for both GRBL FEM and HRBL FEM, which are optimal. Superconvergence phenomenon is also observed. For GRBL FEM, $DG_hu_h$ and $G_hu_h$ converge to the $D^2u$ and $\nabla u$ with convergence rates $\mathcal{O}(h)$ and $\mathcal{O}(h^2)$, respectively. As for the HRBL FEM, the convergence rate of numerical errors in recovered gradient norms are the same as that from the scheme \eqref{discrete weak formulation}; but the convergence rate in $H^2$-seminorm is nearly $\mathcal{O} (h^{2})$, which is one order higher than that of the first scheme \eqref{discrete weak formulation}.

\begin{table}[!h]
\centering
\small{\caption{\emph{{Example $\ref{ex:problem 15}$ -- Numerical results of the scheme \eqref{discrete weak formulation} on cuboid meshes}}}}
\begin{tabular}{lllllllllll}
\hline
 $~\text{DOFs} $ &~~$\|e\|_0$   &order  &~~$|e|_1$  &order  &~~$|e|_{1,r}$&order  &~~$|e|_2$  &order\\
\hline
 $~~~585$ & 1.23E-0 &~      &4.52E-0  &$~$     &3.98E-0   &$~$    &6.47E-0   &$~$\\
 $~~3825$ & 2.56E-1 &~2.26  &2.23E-1  &$~1.03$ &9.71E-1   &$~2.04$&3.05E-0   &$~1.09$\\
 $~27489$ & 6.14E-2 &~2.06  &1.11E-1  &$~1.00$ &2.42E-1   &$~2.01$&1.52E-0   &$~1.01$\\
 $208065$ & 1.53E-2 &~2.00  &5.50E-1  &$~1.00$ &6.05E-2   &$~2.00$&7.59E-1   &$~1.00$\\
\hline
\end{tabular}
\end{table}

\begin{table}[!h]
\centering
\small{\caption{\emph{Example $\ref{ex:problem 15}$ -- Numerical results of the scheme \eqref{discrete weak formulation 2} on cuboid meshes}}}
\begin{tabular}{lllllllllll}
\hline
 $~\text{DOFs} $ &~~$\|e\|_0$   &order  &~~$|e|_1$  &order  &~~$|e|_{1,r}$&order  &~~$|e|_2$  &order\\
\hline
 $~~~585$ & 7.37E-1 &~      &3.24E-0  &$~$     &2.98E-0   &$~$    &4.29E-0   &$~$\\
 $~~3825$ & 9.71E-2 &~2.58  &1.63E-1  &$~1.01$ &4.82E-1   &$~2.13$&9.75E-1   &$~2.13$\\
 $~27489$ & 2.04E-2 &~2.12  &8.01E-1  &$~1.00$  &1.11E-1   &$~2.01$&3.33E-1   &$~1.54$\\
 $208065$ & 4.75E-3 &~2.05  &4.01E-1  &$~1.00$ &2.62E-2   &$~2.00$&6.65E-2   &$~2.32$\\
\hline
\end{tabular}
\end{table}

\subsection{Numerical Experiments of the fully nonlinear Monge--Amp\`{e}re equations }

We shall next test the performance of the proposed numerical scheme \eqref{eqn:weakform2} for solving the fully nonlinear Monge-Amp\`{e}re equation. We will study the impact of the penalty term on stability and accuracy in Examples~\ref{ex:problem 7} and \ref{ex:problem 8}. In all examples, we denote $K$ the number of iterations and $T$ the CPU time cost. We use $\|u^{k+1}-u^{k}\|_2\leq 10^{-8}$ as the stopping criteria of Newton iterations.

\begin{ex}\label{ex:problem 5} A  problem with an exact radial solution.
\end{ex}
In this test, we solve problem \eqref{eqn:para6} on the unit square $\Omega=(0,1)^2$ with the data
$$f=(1+x^2+y^2)e^{(x^2+y^2)/2},\quad g=e^{(x^2+y^2)/2}.$$
This example problem is found in \cite{Benamou, Froese3}. The exact solution of this problem is $u = e^{(x^2+y^2)/2}$.

We take $\sigma=10$ in \eqref{eqn:weakform2} and the numerical results are collected in the Table~6.22. It shows clearly that the proposed numerical method converges with optimal orders of $\mathcal{O}(h^2)$, $\mathcal{O}(h)$, and $\mathcal{O}(h)$ in the $L^2$ norm, $H^1$ seminorm, and $H^2$ seminorm, respectively. The recovered numerical gradient converges with a superconvergence order of $\mathcal{O}(h^2)$ as expected.

We study also the effects of the presentation of high-frequency sinusoidal noise to the data (i.e. the source $f$ and the boundary conditions $g$). Numerical results are shown in Table~6.23. The proposed method yields a solution that is convex except at the boundary although the noisy data is not convex; cf. Figure~6.2. The numerical results show that the noise does not have any effect on the rate of convergence for the presented method.

\begin{table}[h!]
\centering
\small{\caption{\emph{Example~\ref{ex:problem 5}--Numerical results of the scheme \eqref{eqn:weakform2}}}}
\label{Tab:Eg5a}
\begin{tabular}{lllllllllll}
\hline
 $~1/h $&K    &T(s)  &~~$\|e\|_0$   &order  &~~$|e|_1$  &order  &~~$|e|_{1,r}$&order  &~~$|e|_2$  &order\\
\hline
 $~~8 $&5    &0.08  &  6.21E-3 &~     &1.44E-1 &$~   $ &3.08E-2 &$~   $&3.28E-1 &$~    $\\
 $~~16$&5    &0.22  &  1.68E-3 &1.88  &7.12E-2 &$1.02$ &8.10E-3 &$1.93$&1.56E-1 &$~1.08$\\
 $~~32$&5    &0.74  &  4.40E-4 &1.94  &3.55E-2 &$1.01$ &2.04E-3 &$1.99$&7.60E-2 &$~1.04$\\
 $~~64$&5    &2.96  &  1.12E-5 &1.97  &1.77E-2 &$1.00$ &5.11E-4 &$2.00$&3.76E-2 &$~1.01$\\
 $~128$&5    &12.4  &  2.84E-5 &1.99  &8.86E-3 &$1.00$ &1.27E-4 &$2.00$&1.88E-2 &$~1.01$\\
 $~256$&6    &64.1  &  7.12E-6 &1.99  &4.43E-3 &$1.00$ &3.18E-5 &$2.00$&9.36E-3 &$~1.00$\\
\hline
\end{tabular}
\end{table}

\begin{table}[h!]
\centering
\small{\caption{\emph{Example~\ref{ex:problem 5}--Numerical results of the scheme \eqref{eqn:weakform2} with noisy data}}}
\label{Tab:Eg5b}
\begin{tabular}{lllllllllll}
\hline
$~1/h $ &K&T(s)&$~~\|e\|_0$   &order  &~~$|e|_1$  &order  &~~$|e|_{1,r}$&order  &~~$|e|_2$  &order\\
\hline
 $~~~8 $&9&0.09& 3.94E-2 &~     &5.45E-1 &$~   $ &2.18E-1 &$~   $&1.12E-1 &$~    $\\
 $~~16$&7&0.24& 1.11E-2 &1.84  &2.55E-1 &$1.10$ &5.36E-2 &$2.02$&5.37E-1 &$~1.06$\\
 $~~32$&6&0.71& 3.69E-3 &1.59  &1.21E-1 &$1.08$ &1.20E-2 &$2.15$&2.55E-1 &$~1.07$\\
 $~~64$&5&2.92& 1.06E-3 &1.81  &5.94E-2 &$1.03$ &2.83E-3 &$2.09$&1.24E-1 &$~1.05$\\
 $~128$&5&12.2& 2.78E-4 &1.93  &2.96E-2 &$1.01$ &6.86E-4 &$2.04$&6.09E-2 &$~1.01$\\
 $~256$&5&57.1& 7.09E-5 &1.97  &1.48E-2 &$1.00$ &1.69E-4 &$2.02$&3.03E-2 &$~1.00$\\
\hline
\end{tabular}
\end{table}


\begin{figure}[!h]
    \centering
    {\includegraphics[width=0.35\textwidth]{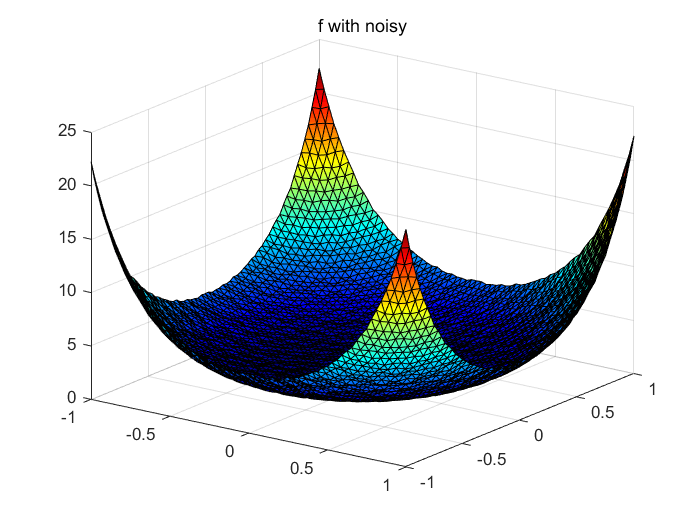}}
    {\includegraphics[width=0.35\textwidth]{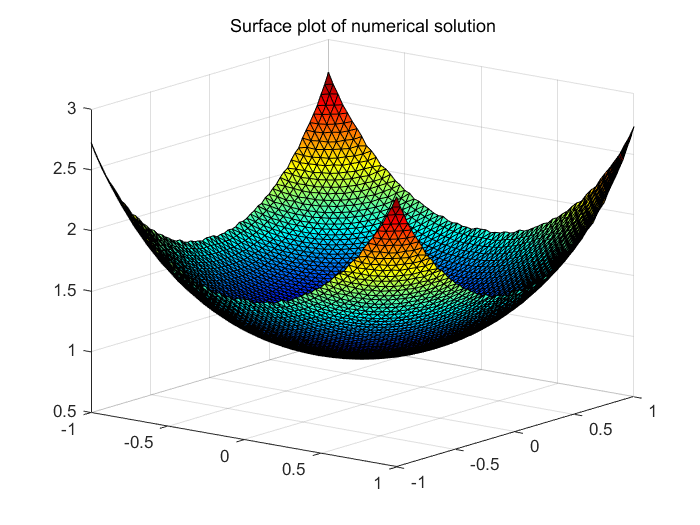}}
\small{\caption{\emph{Example~\ref{ex:problem 5} -- Numerical results with noisy data on a 64$\times$64 grid}}}
\end{figure}

\begin{ex}\label{ex:problem 6} A problem with blow-up at boundary.
\end{ex}
In this test, we choose data such that the exact solution to problem \eqref{eqn:para6} is $u=(x^2+y^2)^{5/3}$. As shown in Example~\ref{ex:problem 3}, $u\in H^{8/3-\tau}$ for arbitrarily small $\tau>0$. The source function $f$ blows up at the boundary.

We take $\Omega=(0,1)^2$ and set $\sigma=2$. The numerical results are illustrated in Table~6.24. Numerical results suggest convergence orders of $\mathcal{O}(h^2)$, $\mathcal{O}(h)$, $\mathcal{O}(h^{2/3})$ in the $L^2$ norm, $H^1$ seminorm, and $H^2$ seminorm, respectively. While the recovered numerical gradient converges with a superconvergence order of $5/3$.
Figure~6.3 shows the profiles of the approximated solution with $\sigma=2$ on a 64$\times$64 grid.

\begin{table}[ht]
\centering
\small{\caption{\emph{Example~\ref{ex:problem 6} -- Numerical results of the scheme \eqref{eqn:weakform2}}}}
\label{Tab:Eg6}
\begin{tabular}{lllllllllll}
\hline
$~1/h $ &K&~~$\|e\|_0$   &order  &~~$|e|_1$  &order  &~~$|e|_{1,r}$&order  &~~$|e|_2$  &order\\
\hline
 $~~8$&6 & 4.15E-3 &~     &8.34E-2 &$~   $ &1.13E-2 &$~   $&2.17E-1 &$~    $\\
 $~16$&6 & 1.10E-3 &1.91  &4.27E-2 &$0.97$ &3.62E-3 &$1.64$&1.37E-2 &$~0.68$\\
 $~32$&7 & 2.78E-4 &1.99  &2.19E-2 &$0.97$ &1.16E-3 &$1.64$&8.66E-2 &$~0.66$\\
 $~64$&7 & 6.52E-5 &2.09  &1.09E-2 &$1.00$ &3.74E-4 &$1.63$&5.52E-2 &$~0.65$\\
 $128$&7 & 1.32E-5 &2.31  &5.22E-3 &$1.06$ &1.25E-4 &$1.60$&3.52E-2 &$~0.65$\\
 $256$&8 & 3.14E-6 &2.07  &2.47E-3 &$1.08$ &4.12E-5 &$1.60$&2.25E-2 &$~0.65$\\
\hline
\end{tabular}
\end{table}

\begin{figure}[!h]
    \centering
    {\includegraphics[width=0.35\textwidth]{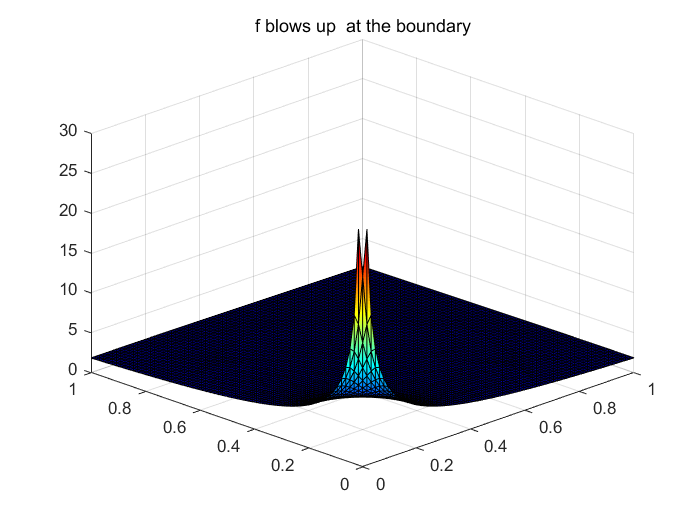}}
    {\includegraphics[width=0.35\textwidth]{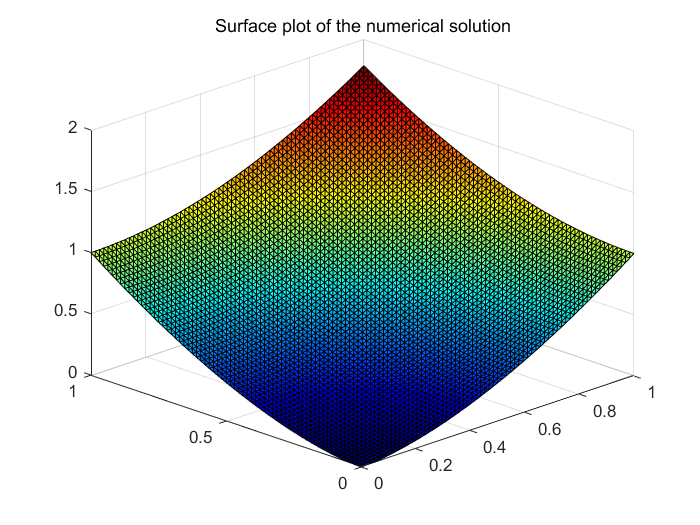}}
\small{\caption{\emph{Example~\ref{ex:problem 6} -- Numerical results with $\sigma=2$ on a 64$\times$64 grid}}}
\end{figure}

\begin{ex}\label{ex:problem 7} Problem in \cite[Example 3]{Chen}.
\end{ex}
In this test, we consider \eqref{eqn:para6} on $\Omega= (-0.5,0.5)^2$ with the data
$$f(x,y)=\text{max}(1-\frac{0.1}{\sqrt{x^2+y^2}},0), \quad g(x,y)=\frac{1}{2}(\sqrt{x^2+y^2}-0.1)^2.$$
The exact solution is given by $u=\frac{1}{2}\text{max}(\sqrt{x^2+y^2}-0.1,0)^2$. The singularity appears along the circle $x^2+y^2=0.1^2$.

We set $\sigma=2$ and $\sigma=0$ in the numerical scheme \eqref{eqn:weakform2}, which represent the cases with and without penalty, respectively. Numerical results are shown in Tables~6.25 and 6.26. It is observed that the absence of penalty leads to unsteady results. Compared with the penalty method, the penalty-free method requires more iterations to converge to the exact solution and has a worse convergence rate.

\begin{table}[h!]
\centering
\small{\caption{\emph{Example~\ref{ex:problem 7} -- Numerical results of the numerical scheme \eqref{eqn:weakform2} with penalty}}}
\label{Tab:Eg7a}
\begin{tabular}{lllllllllll}
\hline
$~1/h $ &K&~~$\|e\|_0$   &order  &~~$|e|_1$  &order  &~~$|e|_{1,r}$&order  &~~$|e|_2$  &order\\
\hline
 $~~8$&~9 & 2.13E-3 &~     &4.42E-2 &$~   $ &1.28E-2 &$~   $&1.84E-1 &$~    $\\
 $~16$&~8 & 6.23E-3 &1.77  &2.24E-2 &$0.98$ &7.19E-3 &$0.83$&1.39E-1 &$~0.40$\\
 $~32$&10 & 3.25E-4 &0.94  &1.15E-2 &$0.96$ &4.22E-3 &$0.77$&1.02E-1 &$~0.46$\\
 $~64$&11 & 1.87E-4 &0.80  &5.94E-3 &$0.96$ &2.44E-3 &$0.79$&7.71E-2 &$~0.39$\\
 $128$&30 & 6.62E-4 &1.49  &2.91E-3 &$1.03$ &1.01E-3 &$1.26$&5.84E-2 &$~0.40$\\
 $256$&17 & 1.88E-5 &1.81  &1.48E-3 &$0.98$ &5.56E-4 &$0.86$&4.95E-2 &$~0.24$\\
\hline
\end{tabular}
\end{table}

\begin{table}[h!]
\centering
\small{\caption{\emph{Example~\ref{ex:problem 7} -- Numerical results of the numerical scheme \eqref{eqn:weakform2} without penalty}}}
\label{Tab:Eg7b}
\begin{tabular}{lllllllllll}
\hline
$~~n $ &~K&~~$\|e\|_0$   &order  &~~$|e|_1$  &order  &~~$|e|_{1,r}$&order  &~~$|e|_2$  &order\\
\hline
 $~~8$&~~8 & 2.46E-3 &~      &4.55E-2 &$~   $ &8.45E-3 &$~   $&1.67E-1 &$~    $\\
 $~16$&~~8 & 5.65E-4 &~2.12   &2.21E-2 &$1.04$ &4.06E-3 &$1.06$&1.17E-1 &$~~0.52$\\
 $~32$&~13 & 1.66E-4 &~1.77   &1.12E-2 &$0.98$ &2.40E-3 &$0.76$&7.94E-2 &$~~0.56$\\
 $~64$&~19 & 8.94E-5 &~0.89   &5.66E-3 &$0.98$ &1.48E-3 &$0.70$&5.98E-2 &$~~0.41$\\
 $128$&124 & 9.11E-5 &-0.03  &3.48E-3 &$0.70$ &1.45E-3 &$0.03$ &7.01E-2 &$-0.23$\\
 $256$&~73 & 8.71E-5 &~0.06   &2.07E-3 &$0.75$ &1.38E-3 &$0.07$&6.74E-2 &$~~0.06$\\
\hline
\end{tabular}
\end{table}

\begin{ex}\label{ex:problem 8} Problem whose solution is a cone.
\end{ex}
In the last test, we choose the data such that the exact solution is a cone.
\begin{equation*}
  u=\sqrt{(x-0.5)^2+(y-0.5)^2}, \quad f=\pi \delta_{(0.5,0.5)}.
\end{equation*}
Following a similar strategy as in \cite{Benamou, Chen, Froese2, Neilan1}, we replace $f$ by its regularized discrete version:
\begin{eqnarray*}
f_h=\left \{
\begin{array}{lll}
~~\pi/(4h^2) \quad \mbox{if}~~|x-0.5|<h~~\mbox{and}~~|y-0.5|<h,\\
~~~~~\quad 0   ~~~~~~\mbox{otherwise} .
\end{array}
\right.
\end{eqnarray*}

We find that the absence of penalty leads to divergent Newton iterations. For example, when we take $\sigma=0$ and $n=64$, the numerical solution does not converge after 500 iterations. But the numerical scheme with penalty reaches the stopping tolerance after 64 iterations. The surface plots of the numerical solution and absolute error are demonstrated in Figure~6.4.


\begin{figure}[!h]
    \centering
    {\includegraphics[width=0.35\textwidth]{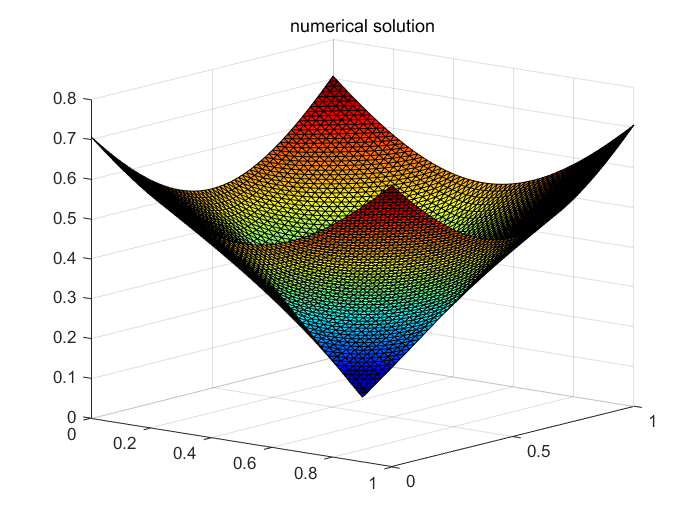}}
    {\includegraphics[width=0.35\textwidth]{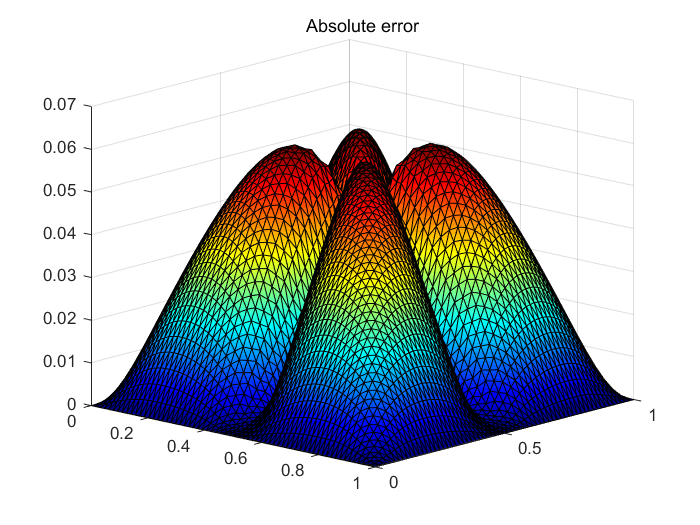}}
\small{\caption{\emph{Example~\ref{ex:problem 8} -- Numerical results with $\sigma=2$ on a 64$\times$64 grid}}}
\end{figure}

\section{Concluding remarks}
In this paper, we present GRBL and HRBL FEMs for second order elliptic equations in non-divergence form. By utilizing the gradient and Hessian recovery operators, we discrete the least square variation in the simplest linear element space. For the GRBL FEM, we prove the stability by adding a rotation. The optimal order of the convergence error is theoretically shown under a discrete $H^2$ seminorm. When coefficients are constants, optimal error estimates in $L^2$ norm and $H^1$ seminorm have also been proven. For the HRBL FEM, optimal convergence in $L^2$ and $H^1$ norms and stability are confirmed from a series of benchmark tests. Finally, the GRBL FEM has been applied to solve the fully nonlinear Monge-Amp\`{e}re equation. Numerical results verify the robustness and the optimal-order convergence.


\bibliographystyle{siam}

\begin{thebibliography}{10}
\bibitem{Adama}
R. Adamas, Sobolev Space. Academic Press, New York (1975).

\bibitem{Ainsworth}
M. Ainsworth and J. T. Oden, A Posteriori Error Estimation in Finite Element Analysis. Wiley Interscience, New York (2000).

\bibitem{Bank1}
R. E. Bank and A. Weiser, Some a posteriori error estimators for elliptic partial differential equations. Math. Comp. 44, 283-301 (1985).

\bibitem{Bank2}
R. E. Bank and J. Xu, Asymptotically exact a posteriori error estimators, Part I: Grid with superconvergence. SIAM J. Numer. Anal. 41, 2294-2312 (2003).

\bibitem{Benamou}
J. D. Benamou, B. D. Froese and A. M. Oberman, Two numerical methods for the elliptic Monge-Amp\`{e}re equation. M2AN Math. Model. Numer. Anal. 44 (4) (2010) 737-758.

\bibitem{Blechschmidt2019}
J. Blechschmidt, R. Herzog, and M. Winkler. Error estimation for second-order pdes in nonvariational form. arXiv preprint arXiv:1909.12676, 2019.

\bibitem{BRW2020}
J. Blechschmidt, R. Herzog and M. Winkler, Error estimation for second-order partial differential equations in nonvariational form. Numer. Methods Partial Differential Equations. 2020. https://doi.org/10.1002/num.22678.

\bibitem{Brenner}
S. C. Brenner, T. Gudi, M. Neilan and L. Y. Sung, A $C^0$ penalty method for the fully nonlinear Monge-Amp\`{e}re equation. Math. Comp. 80 (2011) 1979-1995.

\bibitem{Brenner2020}
S. C. Brenner and E. L. Kawecki, Adaptive $C^0$ interior penalty methods for Hamilton-Jacobi-Bellman equations with Cordes coefficients. J. Comput. Appl. Math. to appear, 2020.

\bibitem{Chen}
Y. G. Chen, J. W. L. Wan and J. Lin, Monotone Mixed Finite Difference Scheme for Monge-Amp\`{e}re equation. J. Sci. Comput. 76 (3) (2018) 1839-1867.

\bibitem{Cordes1956}
H. O. Cordes, \"{U}ber die erste Randwertaufgabe bei quasilinearen Differentialgleichungen zweiter Ordnung in mehr als zwei Variablen. Math. Ann., 131:278-312, 1956.

\bibitem{Costabel}
M. Costabel and M. Dauge, Maxwell and Lam\'{e} eigenvalues on polyhedra. Math. Methods. Appl. Sci. 22 (1999), 243-258.

\bibitem{Dean}
E. J. Dean and R. Glowinski, Numerical solution of the two-dimensional elliptic Monge-Amp\`{e}re equation with Dirichlet boundary conditions: An augmented Lagrangian approach, C. R. Math. Acad. Sci. Paris, 336(2003), 779-784.

\bibitem{Dedner}
A. Dedner and T. Pryer, Discontinuous Galerkin methods for non variational problems. arXiv:1304.2265v1.

\bibitem{Evans}
L. Evans, Partial Differential Equations. American Mathematical Society, Providence, Rhode Island (1998).

\bibitem{Feng}
X. Feng, L. Hennings and M. Neilan, Finite element methods for second order linear elliptic partial differential equations in non-divergence form. Math. Comp. 86 (2017), 2025-2051.

\bibitem{FN2008}
X. Feng and M. Neilan, Vanishing moment method and moment solutions for fully nonlineear second order partial differential equations, J. Sci. Comput. 38 (2008), 74-98.

\bibitem{Feng2018}
X. Feng, M. Neilan, and S. Schnake, Interior penalty discontinuous Galerkin methods for second order linear non-divergence form elliptic PDEs, J. Sci. Comput. 74 (2018), 1651-1676.

\bibitem{Fleming}
W. H. Fleming and H. M. Soner, Controlled Markov Processes and Viscosity Solutions. 2nd ed., Stoch. Model. Appl. Probab. 25, Springer, New York, 2006.

\bibitem{Froese2}
B. D. Froese and A.M. Oberman, Convergent difference schemes for viscosity solutions of the elliptic Monge-Amp\`{e}re equation in dimensions two and higher. SIAM J. Numer. Anal. 49 (4) (2011) 1692-1714.

\bibitem{Froese3}
B. D. Froese and A.M. Oberman, Convergent filtered schemes for the Monge--Amp\`{e}re partial differential equation. SIAM J. Numer. Anal. 51 (1) (2013) 423-444.

\bibitem{Gallistl}
D. Gallistl, Variational Formulation and Numerical Analysis of Linear Elliptic Equations in Nondivergence form with Cord\`{e}s Coefficients. SIAM J. Numer. Anal. 55(2017), 737-757.

\bibitem{Gallist2018}
D. Gallistl, Numerical approximation of planar oblique derivative problems in nondivergence form. Math. Comp, 2018.

\bibitem{Guo}
H. Guo and X. Yang, Polynomial preserving recovery for high frequency wave propagation. J. Sci. Comput. 71(2017), 594-614.

\bibitem{GuoZhang}
H. Guo, Z. Zhang and R. Zhao, Hessian recovery for finite element methods, Math. Comput. 86 (2017) 1671-1692.

\bibitem{GZZZ2016}
H. Guo, Z. Zhang, R. Zhao, and Q. Zou, Polynomial preserving recovery on boundary. J. Comput. Allp. Math. 307 (2016), 119-133.

\bibitem{Guo2017}
H. Guo, Z. Zhang and Q. Zou, A $C^0$ Linear Finite Element Method for Biharmonic Problems. J. Sci. Comp. 74(3) (2018), 1397-1422.

\bibitem{refJiaXu}
Y. Jia, M. Xu, Y. Lin, and D. Jiang, An efficient technique based on least-squares method for fractional integro-differential equations, Alex. Eng. J., (2022), https://doi.org/10.1016/j.aej.2022.08.033.

\bibitem{Kawecki20191}
E. Kawecki, A DGFEM for nondivergence form elliptic equations with Cordes coefcients on curved domains. Numer. Meth. Part. D. E. 35(5)(2019)1717-1744.

\bibitem{Kawecki20192}
E. Kawecki, A discontinuous Galerkin finite element method for uniformly elliptic two dimensional oblique boundary-value problems. SIAM J. Numer. Anal. 57(2)(2019):751-778,

\bibitem{Kawecki20193}
E. Kawecki, O. Lakkis, and T. Pryer. A finite element method for the Monge--Amp\`{e}re equation with transport boundary conditions. arXiv preprint arXiv:1807.03535, 2018.

\bibitem{Kawecki20194}
E. Kawecki and I. Smears, Convergence of adaptive discontinuous Galerkin and $C^0$-interior penalty finite element methods for Hamilton--Jacobi--Bellman and Isaacs equations. 2020. arXiv:2006.07215.

\bibitem{Kawecki20195}
E. Kawecki and I. Smears, Unified analysis of discontinuous galerkin and $C^0$-interior penalty finite element methods for Hamilton-Jacobi-Bellman and Isaacs equations. 2020. arXiv:2006.07202.

\bibitem{Lakkis2019}
O. Lakkis and A. Mousavi, A least-squares galerkin approach to gradient and Hessian recovery for nondivergence-form elliptic equations. 2019. arXiv preprint arXiv:1909.00491.

\bibitem{Lakkisand}
O. Lakkis and T. Pryer, A finite element method for second order nonvariational ellipitic problems. SIAM J. Sci. Comput. 33(2) (2011), 786-801.

\bibitem{Lakkis}
O. Lakkis and T. Pryer, A finite element method for nonlinear elliptic problems. SIAM J. Sci. Comput. 35(4)(2013), 2025-2045.

\bibitem{Loeper}
G. Loeper and F. Rapetti, Numerical solution of the Monge-Amp\`{e}re equation by a Newton's algorithm. C. R. Math. Acad. Sci. Paris, 340(4) (2005b)319-324.

\bibitem{Mu}
L. Mu and X. Ye, A simple finite element method for non-divergence form elliptic equations, Int. J. Numer. Anal. Mod. 14(2) (2017), 306-311.

\bibitem{Naga}
A. Naga and Z. Zhang, The polynomial-preserving recovery for higher order finite element methods in 2D and 3D. Discret. Contin. Dyn. Syst.-Ser. B 5-3 (2005), 769-798.

\bibitem{Neilan1}
M. Neilan, Finite element methods for fully nonlinear second order PDEs based on a discrete Hessian with applications to the Monge-Amp\`{e}re equation. J. Compt. Appl. Math. 263 (2014) 351-369.

\bibitem{Neilan2}
M. Neilan, Quadratic finite element methods for the Monge-Amp\`{e}re equation. J. Sci. Comput. 54(1) (2013) 200-226.

\bibitem{Neilan3}
M. Neilan and M. Wu, Discrete miranda-talenti estimates and applications to linear and nonlinear pdes. J. Compt. Appl. Math. 356 (2019) 358-376.

\bibitem{Nochetto}
R. H. Nochetto and W. Zhang, Discrete ABP estimate and convergence rates for linear elliptic equations in non-divergence form. Found. Comput. Math. 18(3) (2018), 537-593.


\bibitem{Pryer}
T. Pryer, Recovery Methods for Evolution and Nonlinear Problems. D.Phil. thesis, 2010, University of Sussex.

\bibitem{Smears}
I. Smears and E. S\"{u}li, Discontinuous Galerkin finite element approximation of nondivergence form elliptic equations with Cord\`{e}s coefficients. SIAM J. Numer. Anal. 51, (2013), 2088-2106.

\bibitem{Smears1}
I. Smears and E. S\"{u}li. Discontinuous Galerkin finite element approximation of Hamilton-Jacobi-Bellman equations with Cord\`{e}s coefficients SIAM J. Numer. Anal. 52(2) (2014), 993-1016.

\bibitem{Wang}
C. Wang and J. Wang, A primal-dual weak Galerkin finite element method for second order elliptic equations in non-divergence form. Math. Comp. 87 (2018), 515-545.

\bibitem{XuGuoZou}
M. Xu, H. Guo and Q. Zou
, Hessian recovery based finite element methods for the Two-Dimensional Cahn-Hilliard
Equation, J. Comput. Phys., 386 (2019),  524-540.

\bibitem{refXZE3}
M. Xu, L. Zhang, and E. Tohidi, A fourth-order least-squares based reproducing kernel method for one-dimensional elliptic interface problems, Appl. Numer. Math., 162 (2021) 124-136.

\bibitem{refXZE4}
M. Xu, L. Zhang, and E. Tohidi, An efficient method based on least-squares technique for interface problems, Appl. Math. Lett. (2022), https://doi.org/10.1016/j.aml.2022.108475.


\bibitem{Zienkiewicz}
O. C. Zienkiewicz and J. Z. Zhu, The superconvergence patch recovery and a posteriori error estimates part 1: the recovery technique. Int. J. Numer. Methods. Eng. 33(1992), 1331-1364.

\bibitem{Zhang}
Z. Zhang and A. Naga, A new finite element gradient recovery method: superconvergence property. SIAM J. Sci. Comput. 26-4(2005),1192-1213.

\bibitem{Zhu2020}
P. Zhu and X. Wang. A Least Square Based Weak Galerkin Finite Element Method for Second Order Elliptic Equations in Non-Divergence Form. Acta Math. Sci. Ser. B (Engl. Ed.),
40(5) (2020), 1553-1562.

\end{thebibliography}

\end{document}